\pdfoutput=1
\documentclass[a4paper, 11pt, reqno]{amsart}

\usepackage{amssymb}
\usepackage{amsmath}
\usepackage{amsthm}
\usepackage[numbers]{natbib}
\usepackage{mathtools}
\usepackage{enumitem}
\usepackage{calc}
\usepackage{setspace}
\usepackage{graphicx}
\usepackage{tikz-cd}
\usepackage{mathscinet}
\usepackage{hyperref}

\newcounter{dummy}
\numberwithin{dummy}{section}
\newtheorem{thm}[dummy]{Theorem}
\newtheorem{defn}[dummy]{Definition}
\newtheorem{conj}[dummy]{Conjecture}
\newtheorem{lem}[dummy]{Lemma}
\newtheorem{prop}[dummy]{Proposition}
\newtheorem{cor}[dummy]{Corollary}
\numberwithin{equation}{section}

\renewcommand\Re{\operatorname{Re}}
\renewcommand\Im{\operatorname{Im}}

\DeclareMathOperator{\id}{id}
\DeclareMathOperator{\an}{an}
\DeclareMathOperator{\pr}{pr}
\DeclareMathOperator{\Gal}{\mathrm{Gal}}
\DeclareMathOperator{\GL}{\mathrm{GL}}
\DeclareMathOperator{\M}{\mathrm{M}}

\DeclareMathOperator{\Sp}{\mathrm{Sp}}

\DeclareMathOperator{\Aut}{\mathrm{Aut}}
\DeclareMathOperator{\End}{\mathrm{End}}
\DeclareMathOperator{\Hom}{\mathrm{Hom}}
\DeclareMathOperator{\diag}{\mathrm{diag}}
\DeclareMathOperator{\tors}{\mathrm{tors}}
\DeclareMathOperator{\Tr}{\mathrm{Tr}}

\makeatletter
\renewcommand{\@biblabel}[1]{[#1]\hfill}
\makeatother

\begin{document}
\title{Unlikely Intersections between Isogeny Orbits and Curves}

\author{Gabriel A. Dill}
\address{Departement Mathematik und Informatik, Universit\"{a}t Basel, Spiegelgasse~1, CH-4051 Basel}
\email{gabriel.dill@unibas.ch}
\date{\today}

\begin{abstract}
Fix an abelian variety $A_0$ and a non-isotrivial abelian scheme over a smooth irreducible curve, both defined over the algebraic numbers. Consider the union of all images of translates of a fixed finite-rank subgroup of $A_0$, also defined over the algebraic numbers, by abelian subvarieties of $A_0$ of codimension at least $k$ under all isogenies between $A_0$ and some fiber of the abelian scheme. We characterize the curves inside the abelian scheme which are defined over the algebraic numbers, dominate the base curve and potentially intersect this set in infinitely many points. Our proof follows the Pila-Zannier strategy.
\end{abstract}

\maketitle
\tableofcontents

\section{Introduction}\label{sec:Introduction}
Let $K$ be a field of characteristic zero, let $\mathcal{S}$ be a geometrically irreducible smooth curve and let $\mathcal{A} \to \mathcal{S}$ be an abelian scheme over $\mathcal{S}$ of relative dimension $g$, both defined over $K$. The structural morphism will be denoted by $\pi: \mathcal{A} \to \mathcal{S}$ and is smooth and proper. For any (possibly non-closed) point $s$ of $\mathcal{S}$ and any subvariety $\mathcal{V}$ of $\mathcal{A}$, we denote the fiber of $\mathcal{V}$ over $s$ by $\mathcal{V}_s$. The zero section $\mathcal{S} \to \mathcal{A}$ is denoted by $\epsilon$.

We fix an algebraic closure $\bar{K}$ of $K$. All varieties that we consider will be defined over $\bar{K}$, unless explicitly stated otherwise. All varieties will be identified with the set of their closed points over a prescribed algebraic closure of their field of definition. Subvarieties will always be closed. By ``irreducible'', we will always mean ``geometrically irreducible''. If $F$ is any field extension of the field over which the variety $V$ is defined, we will denote the set of points of $V$ that are defined over $F$ by $V(F)$. If $A$ is an abelian variety, we denote by $A_{\tors}$ the set of its torsion points.

We fix an abelian variety $A_0$ of dimension $g$ and a finite set of $\mathbb{Z}$-linearly independent points $\gamma_1, \dots, \gamma_r$ in $A_{0}$. The set can also be empty (i.e. $r=0$). We define
\[ \Gamma = \{ \gamma \in A_0; \mbox{ } \exists N \in \mathbb{N} \mbox{: } N\gamma \in \mathbb{Z}\gamma_1+\hdots+\mathbb{Z}\gamma_r \}, \]
a subgroup of $A_0$ of finite rank (and every subgroup of $A_0$ of finite rank is contained in a group of this form), for us $\mathbb{N} = \{1,2,3,\hdots\}$.

The $(g-k)$-enlarged isogeny orbit of $\Gamma$ (in the family $\mathcal{A}$) is defined as
\begin{multline}
\mathcal{A}_{\Gamma}^{[k]} = \{ p \in \mathcal{A}_s; \mbox{ } s \in \mathcal{S}, \mbox{ } \exists \phi: A_0 \to \mathcal{A}_s \mbox{ isogeny and an abelian subvariety}\\
B_0 \subset A_0 \mbox{ of codimension $\geq k$ such that } p \in \phi(\Gamma+B_0) \}.
\end{multline}
This condition is equivalent to the existence of an isogeny $\psi: \mathcal{A}_s \to A_0$ with $\psi(p) \in \Gamma+B_0$. The isogeny orbit of $\Gamma$ is defined as $\mathcal{A}_{\Gamma} = \mathcal{A}_{\Gamma}^{[g]}$.

Let $\xi$ be the generic point of $\mathcal{S}$. We fix an algebraic closure $\overline{K(\mathcal{S})}$ of $\bar{K}(\mathcal{S})$ and let $\left(\mathcal{A}_\xi^{\overline{K(\mathcal{S})}/\bar{K}},\Tr\right)$ denote the $\overline{K(\mathcal{S})}/\bar{K}$-trace of $\mathcal{A}_\xi$, as defined in Chapter VIII, \S 3 of \cite{MR0106225}, where we consider $\mathcal{A}_\xi$ as a variety over $\overline{K(\mathcal{S})}$ by abuse of notation. We call $\mathcal{A}$ isotrivial if $\Tr\left(\mathcal{A}_\xi^{\overline{K(\mathcal{S})}/\bar{K}}\right) = \mathcal{A}_\xi$. In this article, we investigate the following conjecture, a slightly modified version of Gao's Conjecture 1.2, which he calls the Andr\'{e}-Pink-Zannier conjecture, in \cite{G15}.

\begin{conj} (Modified Andr\'{e}-Pink-Zannier over a curve)\label{conj:APZ}
Suppose that $\mathcal{A} \to \mathcal{S}$ is not isotrivial. Let $\mathcal{V}Ê\subset \mathcal{A}$ be an irreducible subvariety. If $\mathcal{A}_{\Gamma} \cap \mathcal{V}$ is Zariski dense in $\mathcal{V}$, then one of the following two conditions is satisfied:
\begin{enumerate}[label=(\roman*)]
\item The variety $\mathcal{V}$ is a translate of an abelian subvariety of $\mathcal{A}_s$ by a point of $\mathcal{A}_{\Gamma} \cap \mathcal{A}_s$ for some $s \in \mathcal{S}$.
\item We have $\pi(\mathcal{V}) = \mathcal{S}$ and over $\overline{K(\mathcal{S})}$, every irreducible component of $\mathcal{V}_\xi$ is a translate of an abelian subvariety of $\mathcal{A}_\xi$ by a point in $(\mathcal{A}_\xi)_{\tors} + \Tr\left(\mathcal{A}_\xi^{\overline{K(\mathcal{S})}/\bar{K}}(\bar{K})\right)$.
\end{enumerate}
\end{conj}

We need to formulate the conclusion in this somewhat involved manner in order to account for the fact that there can exist abelian subvarieties of $\mathcal{A}_\xi$ and points in $(\mathcal{A}_\xi)_{\tors}$ that are not defined over $\bar{K}(\mathcal{S})$ and that the morphism $\Tr$ isn't necessarily defined over $\bar{K}(\mathcal{S})$. It can be considered one relative version of the Mordell-Lang conjecture, proven for abelian varieties by Vojta \cite{MR1109352}, Faltings \cite{MR1307396} and Hindry \cite{MR969244} and in its most general form by McQuillan in \cite{MR1323985}, in analogy to the relative Manin-Mumford results proven by Masser and Zannier in e.g. \cite{MR2766181}. As we can always assume that $K$ is finitely generated over $\mathbb{Q}$ and then embed it in $\mathbb{C}$, it suffices to prove the conjecture for subfields of $\mathbb{C}$.

\emph{Prima facie}, Gao's conjecture only concerns irreducible subvarieties of the universal family of principally polarized abelian varieties of fixed dimension and fixed sufficiently large level structure. However, we can assume without loss of generality that $\mathcal{A}$ is contained in a suitable universal family $\mathfrak{A}_{g,l}$ corresponding to principally polarized abelian varieties of dimension $g$ with so-called orthogonal level $l$-structure (cf. Sections \ref{sec:preliminaries} and \ref{sec:mainproof}), which reduces Conjecture \ref{conj:APZ} to the case considered by Gao. The condition that the base $\mathcal{S}$ in this situation is a weakly special curve in the moduli space seems to be missing in our formulation of the conjecture, but it follows directly from Orr's Theorem 1.2 in \cite{MR3377393} that Conjecture \ref{conj:APZ} can be further reduced to this case. The conjecture is stronger than Gao's in that it involves a subgroup of rank possibly larger than $1$ and doesn't demand that the isogenies are polarized. It is weaker in that the base variety $\mathcal{S}$ is assumed to be a curve.

Gao showed in Section 8 of \cite{G15} that Conjecture \ref{conj:APZ} follows from Pink's Conjecture 1.6 in \cite{MR2166087} in the more general setting of generalized Hecke orbits in mixed Shimura varieties, where it is enough to assume Pink's conjecture for all fibered powers of universal families of principally polarized abelian varieties of fixed dimension and fixed, sufficiently large level structure. By Theorem 3.3 in \cite{PUnpubl}, Conjecture 1.6 in \cite{MR2166087} is a consequence of Pink's even more general Conjecture 1.1 in \cite{PUnpubl} on unlikely intersections in mixed Shimura varieties. If $\Gamma$ has rank zero, Conjecture \ref{conj:APZ} is contained in a special-point conjecture of Zannier (see \cite{G15}, Conjecture 1.4).

Progress towards Conjecture \ref{conj:APZ} has only been made if $\mathcal{V} = \mathcal{C}$ is a curve or if the rank of $\Gamma$ is zero. Furthermore, many results are confined to the case where $K$ is a number field. Lin and Wang have proved the conjecture for $K$ a number field, $\mathcal{V}$ a curve, $\Gamma$ finitely generated and $A_0$ simple (Theorem 1.1 in \cite{MR3383643}). Habegger has proved it for $K$ a number field, $\Gamma$ of rank zero and $\mathcal{A}$ a fibered power of a non-isotrivial elliptic scheme (Theorem 1.2 in \cite{MR3181568}). Pila has proved it for arbitrary $K$, $\Gamma$ of rank zero and $\mathcal{A}$ inside a product of elliptic modular surfaces (Theorem 6.2 in \cite{MR3164515}). Gao has proved it for arbitrary $K$ and $\Gamma$ of rank zero (Theorem 1.5 in \cite{G15}) as well as for arbitrary $K$, $\mathcal{V}$ a curve and $\Gamma$ of rank at most one, but in this case he has to fix polarizations of $A_0$ and $\mathcal{A}$ and assume that the isogenies are polarized (Theorem 1.6 in \cite{G15}).

From now on, we will always assume that $K \subsetÊ\mathbb{C}$ is a number field and take as $\bar{K} = \bar{\mathbb{Q}}$ its algebraic closure in $\mathbb{C}$. We expect however that Theorem \ref{thm:main} can be generalized to the transcendental case in the same way as Gao's by use of the Moriwaki height instead of the Weil height together with specialization arguments.

The purpose of this paper is twofold: First, we prove Conjecture \ref{conj:APZ} in Theorem \ref{thm:main} if $K$ is a number field and $\mathcal{V} = \mathcal{C}$ is a curve. Second, we investigate what happens when $\mathcal{C} \cap \mathcal{A}_{\Gamma}^{[k]}$ is infinite for some arbitrary $k \in \{0,\hdots,g\}$. Here, the case $k = g$ corresponds to Conjecture \ref{conj:APZ}. If $k < g$, the condition is weaker (if $k=0$, it is void), so we expect a weaker conclusion. We prove the strongest possible conclusion in Theorem \ref{thm:mainmain}, of which Theorem \ref{thm:main} thus becomes a special case.

The problem of intersecting a fixed subvariety with algebraic subgroups originates in works of Bombieri-Masser-Zannier \cite{MR1728021} and Zilber \cite{MR1875133} for powers of the multiplicative group. The analogous problem in a fixed abelian variety has also been the object of much study; we just mention the work of Habegger and Pila \cite{MR3552014}, from which we use several results in our proof. The intersection of a subvariety of a fixed abelian variety with translates of abelian subvarieties by points of a subgroup of finite rank has been studied by R\'{e}mond in e.g. \cite{MR2311666}. While there has been intensive study of unlikely intersections between a curve in an abelian scheme and flat algebraic subgroup schemes, culminating in the article by Barroero and Capuano \cite{BC18}, ours seems to be the first result that combines intersecting with positive-dimensional algebraic subgroups with an isogeny condition on the fiber.

We can now state our main results. Recall that $\mathcal{S}$ is a smooth irreducible curve and $\mathcal{A} \to \mathcal{S}$ is an abelian scheme, both defined over $K$, while $\mathcal{C} \subset \mathcal{A}$ is a closed irreducible curve, defined over $\bar{\mathbb{Q}}$, $A_0$ is an abelian variety defined over $\bar{\mathbb{Q}}$, $\gamma_1,\hdots,\gamma_r \in A_0(\bar{\mathbb{Q}})$ and $\Gamma \subset A_0$ is the subgroup of all $\gamma \in A_0$ such that $N\gamma \in \mathbb{Z}\gamma_1+\hdots+\mathbb{Z}\gamma_r$ for some $N \in \mathbb{N}$.

\begin{thm}\label{thm:mainmain}
Suppose that $\mathcal{A} \to \mathcal{S}$ is not isotrivial. If $\mathcal{A}_{\Gamma}^{[k]} \cap \mathcal{C}$ is infinite and $\pi(\mathcal{C})=\mathcal{S}$, then $\mathcal{C}$ is contained in an irreducible subvariety $\mathcal{W}$ of $\mathcal{A}$ of codimension $\geq k$ with the following property: Over $\overline{\bar{\mathbb{Q}}(\mathcal{S})}$, every irreducible component of $\mathcal{W}_\xi$ is a translate of an abelian subvariety of $\mathcal{A}_\xi$ by a point in $(\mathcal{A}_\xi)_{\tors} + \Tr\left(\mathcal{A}_\xi^{\overline{\bar{\mathbb{Q}}(\mathcal{S})}/\bar{\mathbb{Q}}}(\bar{\mathbb{Q}})\right)$.
\end{thm}

\begin{thm}\label{thm:main}
Suppose that $\mathcal{A} \to \mathcal{S}$ is not isotrivial. If $\mathcal{A}_{\Gamma} \cap \mathcal{C}$ is infinite, then one of the following two conditions is satisfied:
\begin{enumerate}[label=(\roman*)]
\item The curve $\mathcal{C}$ is a translate of an abelian subvariety of $\mathcal{A}_s$ by a point of $\mathcal{A}_{\Gamma} \cap \mathcal{A}_s$ for some $s \in \mathcal{S}$.
\item The zero-dimensional variety $\mathcal{C}_\xi$ is contained in $(\mathcal{A}_\xi)_{\tors} + \Tr\left(\mathcal{A}_\xi^{\overline{\bar{\mathbb{Q}}(\mathcal{S})}/\bar{\mathbb{Q}}}(\bar{\mathbb{Q}})\right)$.
\end{enumerate}
\end{thm}

From Theorem \ref{thm:main}, we can deduce the following corollary:

\begin{cor}\label{cor:dacorollary}
Let $A_{g,l}$ be the moduli space of principally polarized abelian varieties of dimension $g$ with orthogonal level $l$-structure as defined in Section \ref{sec:preliminaries} and $l$ sufficiently large and let $A$ and $B$ be abelian varieties with $\dim B = g$. Let $CÊ\subset A_{g,l} \times A$ be a closed irreducible curve and let $\pr_1: C \to A_{g,l}$ and $\pr_2: C \to A$ be the canonical projections. Let $\Gamma' \subset A$ be a subgroup of finite rank and let $\Sigma \subset A_{g,l}$ be the set of $sÊ\in A_{g,l}$ corresponding to abelian varieties that are isogenous to $B$. If $C \cap (\Sigma \times \Gamma')$ is infinite, then either $\pr_1$ or $\pr_2$ is constant.
\end{cor}

We thereby prove Conjecture 1.7 of Buium and Poonen in \cite{MR2507742}: If $S$ is a modular curve or a Shimura curve, then a Zariski open subset $S'$ of $S$ has a moduli interpretation which yields a quasi-finite forgetful modular morphism from $S'$ to the coarse moduli space $A_g$ of principally polarized abelian varieties of dimension $g \in \{1,2\}$. Similarly, we have a quasi-finite morphism $A_{g,l} \to A_g$. We can then form the curve $S' \times_{A_g} A_{g,l}$, which admits quasi-finite morphisms to $S'$ and $A_{g,l}$, and reduce the conjecture to Corollary \ref{cor:dacorollary}. The conjecture of Buium and Poonen has been proven independently by Baldi in \cite{B18} through the use of equidistribution results. He was also able to replace $\Gamma'$ by a fattening $\Gamma'_{\epsilon}$ for some $\epsilon > 0$ (see \cite{B18} for the definition of $\Gamma'_{\epsilon}$). Such an extension seems to lie outside the reach of our methods though.

The proof of Theorem \ref{thm:mainmain} uses point counting and o-minimality and in particular a later refinement of the theorem of Pila-Wilkie on rational points on definable sets in \cite{MR2228464}. In applying this result to problems of unlikely intersections in diophantine geometry, we follow the standard strategy as devised by Zannier for the new proof of the Manin-Mumford conjecture by Pila and him in \cite{MR2411018}. It is described in Zannier's book \cite{MR2918151}. In Section \ref{sec:preliminaries}, we introduce some notation and make several reduction steps.

In Sections \ref{sec:heightbounds} and \ref{sec:galoisorbitbounds}, we bound the ``height'' of all important quantities from above in terms of the degree of the varying point $p = \phi(q) \in \mathcal{A}^{[k]}_{\Gamma} \cap \mathcal{C}$ over the fixed number field $K$. The main new ideas of this article are to be found in these two sections. In order to treat non-polarized isogenies, we extend a result by Orr to show that the isogeny $\phi$ between $A_0$ and $\mathcal{A}_s$ can be chosen such that certain associated quantities are bounded in the required way -- first of all, we apply the isogeny theorem of Masser-W\"ustholz to show that the degree of the isogeny can be bounded in this way. As a consequence of our extension of Orr's result we can then bound the height of $q$ for this choice of $\phi$. (After maybe enlarging $\Gamma$, we can fix for each $s \in \mathcal{S}$ such that $A_0$ and $\mathcal{A}_s$ are isogenous one choice of isogeny -- see Lemma \ref{lem:technicallemma}.)

We bound the degree of the smallest translate of an abelian subvariety of $A_0^{r+1}$ by a torsion point that contains $(q,\gamma_1,\hdots,\gamma_r)$ through an application of a proposition by Habegger and Pila. Using this and a lemma of R\'emond, we can then write $q = \gamma + b$ with $\gamma \in \Gamma$ of controlled height and $b$ in an abelian subvariety of controlled codimension and degree. If $N$ is the smallest natural number such that $N\gamma \in \bigoplus_{i=1}^{r} \mathbb{Z}\gamma_i$, we finally bound $N$ by applying a lemma of Habegger and Pila, some elementary diophantine approximation and lower height bounds on abelian varieties due to Masser.

In Section \ref{sec:ominimality}, we give a brief introduction to o-minimal structures in as much depth as is necessary to state a variant of the Pila-Wilkie theorem, due to Habegger and Pila,  on ``semirational" points of bounded height.

In Section \ref{sec:uniformization}, the definability in a suitable o-minimal structure of the analytic uniformization map associated to our abelian scheme is shown, when restricted to some fundamental domain, by use of a theorem of Peterzil-Starchenko. In Section \ref{sec:functionaltranscendence}, we record the necessary algebraic independence result of ``logarithmic Ax'' type by Gao, which generalizes work by Andr\'{e} in \cite{MR1154159} and by Bertrand in \cite{MR2649338}.

Finally, we put all the pieces together in Section \ref{sec:mainproof} and prove Theorem \ref{thm:mainmain}, Theorem \ref{thm:main} and Corollary \ref{cor:dacorollary}.

\section{Preliminaries and Notation}\label{sec:preliminaries}
For a rational number $\alpha=\frac{a}{b}$ with $a \in \mathbb{Z}$, $b \in \mathbb{N}$ and $\gcd(a,b)=1$, we define its affine height $H(\alpha)=\max\{|a|,|b|\}$. We will fix once and for all a square root of $-1$ inside $\mathbb{C}$ that we denote by $\sqrt{-1}$ -- this yields maps $\Re: \mathbb{C} \to \mathbb{R}$ and $\Im: \mathbb{C} \toÊ\mathbb{R}$ in the usual way. For an integral domain $R$, we denote the space of $m\times n$-matrices with entries in $R$ by $\M_{m\times n}(R)$. We write $\M_n(R)$ for $\M_{n \times n}(R)$. For a matrix $A=(a_{ij}) \in \M_{n}(\mathbb{Q})$, we define its height $H(A)=\max_{i,j}H(a_{ij})$. The complex conjugate of a matrix $A$ with complex entries will be denoted by $\overline{A}$ and the transpose by $A^{t}$. The $n$-dimensional identity matrix will be denoted by $E_n$. The row-sum norm of a matrix $A \in \M_{m\times n}(\mathbb{C})$ will be denoted by $\lVert A \rVert$. For a vector $v=(v_1,\hdots,v_n)^t \in \mathbb{C}^n$, we will write $\lVert v \rVert$ for $\max_{j=1,\hdots,n} |v_j|$. Note that
\[ \lVert A \rVert = \max_{vÊ\neq 0}{\frac{\lVert Av\rVert}{\lVert v \rVert}}Ê\]
for all $A \in \M_{m \times n}(\mathbb{C})$. Vectors will always be column vectors. By applying $\Re$ and $\Im$ to each entry, we obtain maps from $\M_n(\mathbb{C})$ to $\M_n(\mathbb{R})$ that by abuse of language will also be called $\Re$ and $\Im$.

If $A$ is an arbitrary abelian variety over an arbitrary field, we denote its dual abelian variety by $\hat{A}$. If $\phi: A \to B$ is an isogeny, the dual isogeny will be denoted by $\hat{\phi}: \hat{B} \to \hat{A}$.

For our proof of Theorem \ref{thm:main}, we will restrict ourselves in the following sections to subfamilies of the universal family $\mathfrak{A}_{g,l} \to A_{g,l}$ of principally polarized abelian varieties with so-called orthogonal level $l$-structure for a natural number $l \geq 16$ which is divisible by $8$ and a perfect square and identify $\pi$ and $\epsilon$ with the natural projection and zero section of that family. If $\mathbb{H}_g$ denotes the Siegel upper half space in dimension $g$ (i.e. symmetric matrices in $\M_g(\mathbb{C})$ with positive definite imaginary part), then $\mathfrak{A}_{g,l}$ is a quotient of $\mathbb{H}_g \times \mathbb{C}^g$ by the semidirect product of the congruence subgroup
\begin{align*}
G(l,2l) = \left\{ M = \begin{pmatrix} A & B \\ C & D \end{pmatrix} \in \Sp_{2g}(\mathbb{Z})\mbox{; } M \equiv E_{2g} \mod l, \right. \\ \left. \diag(A B^t) \equiv \diag(C D^t) \equiv 0 \mod 2l\right\}
\end{align*}
of $\Sp_{2g}(\mathbb{Z})$ with $\mathbb{Z}^{2g}$, where $\diag$ denotes the diagonal of a matrix. We will show at the end of our work in Section \ref{sec:mainproof} how to deduce the result for arbitrary families.

The group $G(l,2l)$ is the same as the group $\Gamma(l,2l)$ defined on p. 422 of \cite{MR1207211}. Let $G_{lE_g}(lE_g)_{0}$ be defined as in Section 8.9 of \cite{MR2062673}. Then there is an isomorphism from $G(l,2l)$ to $G_{lE_g}(lE_g)_{0}$ given by sending $M$ to $\begin{pmatrix} E_g & 0 \\ 0 & l^{-1}E_g\end{pmatrix}M \begin{pmatrix} E_g & 0 \\ 0 & lE_g\end{pmatrix}$ -- see \cite{MR2062673}, Section 8.8 and 8.9, and note that $l$ is even.

The group law on the semidirect product is given by $(M',z')(M,z)=(M'M,z'+(M')^{-t}z)$ and the action of the group is given by
\[ \left(\begin{pmatrix} A & B \\ C & D\end{pmatrix},\begin{pmatrix} m \\ n\end{pmatrix}\right)(\tau,z) = (\tau',(C\tau+D)^{-t}z+\tau'm+n),\]
where
\[ \tau' = \begin{pmatrix} A & B \\ C & D\end{pmatrix}[\tau] := (A\tau+B)(C\tau+D)^{-1}.\]
The action of course extends to an action of $\Sp_{2g}(\mathbb{R}) \ltimes \mathbb{R}^{2g}$ (with the same group law) and then also restricts to the usual action of $\Sp_{2g}(\mathbb{R})$ on $\mathbb{H}_g$. If $M \in \Sp_{2g}(\mathbb{R})$ and $\tau \in \mathbb{H}_g$, we will denote this last action as above by $M[\tau]$ to avoid confusion with ordinary matrix multiplication.

By applying Proposition 8.2.5 in \cite{MR2062673} and Cartan's Expos\'{e} 11 in Volume 2 of \cite{MR0103988}, we see that our universal family is a complex analytic space because the group action is proper and discontinuous -- Proposition 8.2.5 of \cite{MR2062673} only says that the action of $G(l,2l)$ on $\mathbb{H}_g$ is proper and discontinuous, but this quickly implies the same for the action of its semidirect product with $\mathbb{Z}^{2g}$ on $\mathbb{H}_gÊ\times \mathbb{C}^g$. However, the universal family is in fact a quasi-projective variety, defined over $\mathbb{Q}$. In the following proposition, we recall some well-known facts about it.

\begin{prop}\label{prop:universalfamily}
There exist holomorphic maps
\[\exp: \mathbb{H}_g \times \mathbb{C}^g  \to \mathbb{P}^{l^g-1}(\mathbb{C}) \times \mathbb{P}^{l^g-1}(\mathbb{C})\]
and $\iota: \mathbb{H}_g \to \mathbb{P}^{l^g-1}(\mathbb{C})$ with the following properties:
\begin{enumerate}[label=(\roman*)]
\item There is a commutative diagram
\[
\begin{tikzcd}
\mathbb{H}_g \times \mathbb{C}^g \arrow{r}{\exp} \arrow[swap]{d}{} & \arrow{d}{} \mathbb{P}^{l^g-1}(\mathbb{C}) \times \mathbb{P}^{l^g-1}(\mathbb{C}) \\
\mathbb{H}_g \arrow{r}{\iota} & \mathbb{P}^{l^g-1}(\mathbb{C})
\end{tikzcd},
\]
where the vertical maps are projections to the first factor.
\item We have $\exp(\tau,z) = \exp(\tau',z')$ if and only if $(\tau,z)$, $(\tau',z')$ lie in the same $G(l,2l) \ltimes \mathbb{Z}^{2g}$-orbit and $\exp$ descends to an analytic embedding of the quotient. Similarly, we have $\iota(\tau) = \iota(\tau')$ if and only if $\tau$, $\tau'$ lie in the same $G(l,2l)$-orbit and $\iota$ descends to an analytic embedding of the quotient.
\item The images $\exp(\mathbb{H}_g \times \mathbb{C}^g)$ and $\iota(\mathbb{H}_g)$ are locally closed with respect to the Zariski topology in $\mathbb{P}^{l^g-1}(\mathbb{C})Ê\times \mathbb{P}^{l^g-1}(\mathbb{C})$ and $\mathbb{P}^{l^g-1}(\mathbb{C})$ respectively. They are irreducible smooth varieties, defined over $\mathbb{Q}$.
\item $\exp(\mathbb{H}_g \times \mathbb{C}^g) \to \iota(\mathbb{H}_g)$ is an abelian scheme, defined over $\mathbb{Q}$, with zero section $p \mapsto (p,p)$.
\item $\exp(\tau,\cdot)$ is a surjective group homomorphism from $\mathbb{C}^g$ to $\exp(\{\tau\} \times \mathbb{C}^g)$ with kernel $\Omega_\tau \mathbb{Z}^{2g}$, where
\[ \Omega_\tau = \begin{pmatrix}\tau & E_g\end{pmatrix}.\]
\item The very ample line bundle on $\exp(\{\tau\} \times \mathbb{C}^g)$ that is induced by this embedding is the $l$-th tensor power of a symmetric ample line bundle. Under the uniformization $\exp(\{\tau\} \times \mathbb{C}^g) \simeq \mathbb{C}^g/\Omega_{\tau}\mathbb{Z}^{2g}$ given by $\exp$, the Hermitian form on $\mathbb{C}^g$ induced by this second line bundle is given by the matrix $(\Im \tau)^{-1}$.
\end{enumerate}
\end{prop}

\begin{proof}
We can explicitly give the maps, using the classical theta functions. For this, we define
\[ \theta[a,b](\tau,z) = \sum_{m \in \mathbb{Z}^g}{\exp(\pi \sqrt{-1} (m+a)^t \tau (m+a)+2\pi \sqrt{-1} (m+a)^t (z+b))} \]
for $\tau \in \mathbb{H}_g$, $z \in \mathbb{C}^g$ and $a,b \in \mathbb{Q}^g$. For $c \in \mathbb{Q}^g$ and $(\tau,z) \in \mathbb{H}_g \times \mathbb{C}^g$, we put
\[ \theta_{c}(\tau,z) = \theta[c,0](\tau,z).\]

We then define
\[ \phi(\tau,z) = [\theta_{c_0}(\tau,z) : \hdots : \theta_{c_{l^g-1}}(\tau,z)]\]
and $\iota(\tau) = \phi(l\tau,0)$ as well as
\[ \exp(\tau,z) = (\phi(l\tau,0),\phi(l\tau,lz)),\]
where the $c_i$ run over the set $\{0,\frac{1}{l},\hdots,1-\frac{1}{l}\}^{g}$ ($i=0,\hdots,l^g-1$).

Property (i) now follows directly from the definitions. For property (ii), we refer to Chapter 8 of \cite{MR2062673} and Chapter V of \cite{MR0325625}. Note that due to the above-mentioned isomorphism between $G_{lE_g}(lE_g)_{0}$ and $G(l,2l)$ $\tau$ and $\tau'$ lie in the same $G(l,2l)$-orbit if and only if $l\tau$ and $l\tau'$ lie in the same $G_{lE_g}(lE_g)_{0}$-orbit. It follows from \cite{MR1580111}, \S 1, that the actions are free, so the quotient maps are covering maps. The map $\iota$ descends to a proper map from the quotient to its image, since over every point of its image lies exactly one point of the normalization of the closure of the image by \cite{MR610482}. One can use Theorem 4.5.1 of \cite{MR2062673} to show that not only $\iota$, but also $\exp$ descends to an analytic embedding.

For properties (iii) and (iv), see Section 3 of \cite{MR1207211} and the references given there, in particular \cite{MR0219541}. Smoothness and irreducibility follow from the fact that the quotients are connected complex analytic manifolds. Property (v) follows from (ii) and the choice of zero section.

Property (vi) follows by computing the factor of automorphy of the embedding $\exp(\tau,\cdot)$ of $\mathbb{C}^g/\Omega_\tau \mathbb{Z}^{2g}$: An elementary computation shows that
\[\theta_c(l\tau,l(z+m+\tau n)) = \exp(-\pi \sqrt{-1}ln^t\tau n-2\pi \sqrt{-1}ln^tz)\theta_c(l\tau,lz)\]
for all $c \in \{0,\frac{1}{l},\hdots,1-\frac{1}{l}\}^g$ and all $m,n \in \mathbb{Z}^g$. By Remark 8.5.3(d) in \cite{MR2062673}, this factor of automorphy belongs to the $l$-th tensor power of a symmetric ample line bundle that under the given uniformization is associated to the Hermitian form given by $(\Im \tau)^{-1}$ on $\mathbb{C}^g$.
\end{proof}

Using the proposition, we may identify $\exp(\mathbb{H}_g \times \mathbb{C}^g)$ and $\iota(\mathbb{H}_g)$ with $\mathfrak{A}_{g,l}(\mathbb{C})$ and $A_{g,l}(\mathbb{C})$ and use $\mathfrak{A}_{g,l}$ and $A_{g,l}$ for the corresponding quasiprojective varieties, defined over $\mathbb{Q}$. We will denote the Zariski closures in $\mathbb{P}^{l^g-1}$ and $\mathbb{P}^{l^g-1}\times\mathbb{P}^{l^g-1}$ of these varieties by $\overline{\mathfrak{A}_{g,l}}$ and $\overline{A_{g,l}}$ respectively; these are (usually highly singular) projective varieties, also defined over $\mathbb{Q}$. The projection from $\mathbb{P}^{l^g-1}\times\mathbb{P}^{l^g-1}$ onto the first factor yields a morphism $\pi: \overline{\mathfrak{A}_{g,l}} \to \overline{A_{g,l}}$. The embedding from the proposition yields very ample line bundles $\mathcal{L}$ on $\overline{\mathfrak{A}_{g,l}}$ and $L$ on $\overline{A_{g,l}}$.

From now on, we assume that $\mathcal{S} \subset A_{g,l}$ is an irreducible, smooth, locally closed curve (not necessarily closed in $A_{g,l}$), $\mathcal{A} = \pi^{-1}(\mathcal{S})$ and $\mathcal{C} \subset \mathcal{A}$ is an irreducible closed curve. We denote by $\overline{\mathcal{C}}$ and $\overline{\mathcal{S}}$ the Zariski closures of $\mathcal{C}$ and $\mathcal{S}$ in $\overline{\mathfrak{A}_{g,l}}$ and $\overline{A_{g,l}}$ respectively. The abelian scheme $\mathcal{A}Ê\to \mathcal{S}$ and the curve $\overline{\mathcal{S}}$ are defined over $K$.

After maybe enlarging $K$, we can and will assume without loss of generality that $A_0$, the addition morphism $A_0 \times A_0 \to A_0$, the inversion morphism $A_0 \to A_0$, $\mathcal{C}$ and $\overline{\mathcal{C}}$ are defined over $K$ and that $A_0$ is principally polarized. For this, we might have to replace $A_0$ by an isogenous abelian variety and $\Gamma$ by its pre-image under the corresponding isogeny. This doesn't change the isogeny orbit, so doesn't change the statement we want to prove.

We fix a symmetric ample line bundle $L_0$ which gives us a principal polarization on $A_0$ and fix once and for all a uniformization $\mathbb{C}^g/\Omega_{\tau_0}\mathbb{Z}^{2g}$ of $A_0(\mathbb{C})$ such that the Hermitian form on $\mathbb{C}^g$ associated to $L_0$ is given by $(\Im \tau_0)^{-1}$, $\Omega_{\tau_0} = \left(\begin{smallmatrix} \tau_0 & E_g\end{smallmatrix}\right)$ and $\tau_0$ lies in the Siegel fundamental domain (see Definition \ref{defn:siegelfundamentalset}). We denote the corresponding map $\mathbb{C}^g \to A_0(\mathbb{C})$ by $\exp_0$. Using Weil's Height Machine (see \cite{MR1745599}, Theorem B.3.2 and B.3.6), we also get a (logarithmic projective) height $h_{A_0}=h_{A_0,L_0}$ on $A_0$. With the usual construction due to N\'{e}ron and Tate (see \cite{MR1745599}, Theorem B.5.1) we then obtain a canonical height $\widehat{h}_{A_0}$ on $A_0$.

After maybe enlarging $K$ again, we can assume that $L_0$ is defined over $K$, $\gamma_1, \hdots, \gamma_r \in A_0(K)$, and every endomorphism of $A_0$ is defined over $K$. Since the endomorphism ring of $A_0$ is finitely generated as a $\mathbb{Z}$-module, we may assume that $\Gamma$ is mapped into itself by every endomorphism of $A_0$ by enlarging $\Gamma$ if necessary (which only makes Theorems \ref{thm:mainmain} and \ref{thm:main} stronger). We will generally assume that $r \geq 1$ for simplicity -- one can either ensure this by enlarging $\Gamma$ and $K$ or one can check that our proof also works mutatis mutandis if $r=0$.

The line bundle $L$ restricts to a very ample line bundle $L_{\overline{\mathcal{S}}}$ on $\overline{\mathcal{S}}$. For each $s \in \mathcal{S}$, the restriction of $\mathcal{L}$ to $\mathcal{A}_s$ is a very ample symmetric line bundle $\mathcal{L}_s$ by Proposition \ref{prop:universalfamily}(vi). From the embeddings into projective space by theta functions, we directly obtain associated heights $h_{\overline{\mathcal{S}}}$ on $\overline{\mathcal{S}}$ and $h_s$ on $\mathcal{A}_s$ ($s \in \mathcal{S}$) as well as a canonical height $\widehat{h}_s$ on $\mathcal{A}_s$.

The following technical lemma shows that for each $sÊ\in \mathcal{S}$ we can fix an isogeny $\phi_s$ in the definition of $\mathcal{A}_{\Gamma}^{[k]}$.

\begin{lem}\label{lem:technicallemma}
For each $s \in \mathcal{S}$ such that $\mathcal{A}_{s}$ and $A_0$ are isogenous, fix an isogeny $\phi_s: A_0 \to \mathcal{A}_s$. For $\Gamma$ as above, we have
\begin{multline}
\mathcal{A}_{\Gamma}^{[k]} = \{ p \in \mathcal{A}_s; \mbox{ } s \in \mathcal{S}, \mbox{$\mathcal{A}_s$ and $A_0$ isogenous and there exists an abelian}\\
\mbox{subvariety } B_0 \subset A_0 \mbox{ of codimension $\geq k$ such that } p \in \phi_s(\Gamma+B_0) \}.
\end{multline}
\end{lem}
\begin{proof}
We prove the non-trivial inclusion ``$\subset$". Suppose that $p \in \mathcal{A}_{\Gamma}^{[k]}$. Then $p$ lies in some $\mathcal{A}_s$ ($s \in \mathcal{S})$ such that $\mathcal{A}_{s}$ and $A_0$ are isogenous. By definition, there is an isogeny $\phi: A_0 \to \mathcal{A}_s$, an abelian subvariety $B_0$ of $A_0$ of codimension $\geq k$ and $\gamma \in \Gamma$ such that $p \in \phi(\gamma+B_0)$.

We denote by $\tilde{\phi_s}$ the isogeny from $\mathcal{A}_s$ to $A_0$ such that $\phi_s \circ \tilde{\phi_s}$ is multiplication by $\deg \phi_s$ on $\mathcal{A}_s$. Then $\chi = \tilde{\phi_s} \circ \phi$ is an endomorphism of $A_0$ and $\phi_s \circ \chi = (\deg \phi_s)\phi$.

We choose $\tilde{\gamma} \in A_0$ with $(\deg \phi_s)\tilde{\gamma} = \gamma$ and get
\[ p \in \phi(\gamma+B_0)=\phi((\deg \phi_s)\tilde{\gamma})+\phi((\deg \phi_s)B_0) = \phi_s(\chi(\tilde{\gamma}))+\phi_s(\chi(B_0)).\]

We show that $\chi(\tilde{\gamma}) \in \Gamma$ (so $p \in \phi_s(\Gamma) + \phi_s(\chi(B_0))$). Since $(\deg \phi_s)\tilde{\gamma}=\gamma \in \Gamma$, it follows that $\tilde{\gamma}$ lies in $\Gamma$ as well. By our assumption above, $\Gamma$ is mapped into itself by $\chi$. Hence, $\chi(\tilde{\gamma})$ belongs to $\Gamma$ as desired and the lemma follows, since $\chi(B_0)$ is again an abelian subvariety of $A_0$ of codimension $\geq k$.
\end{proof}

We take $\phi_s$ as an isogeny between $\mathcal{A}_{s}$ and $A_0$ of minimal degree, i.e. there exists no isogeny $\psi: A_0 \to \mathcal{A}_s$ of degree less than $\degÊ\phi_s$. By Th\'{e}or\`{e}me 1.4 of Gaudron-R\'{e}mond in \cite{MR3263028}, which improves a theorem of Masser-W\"ustholz (\cite{MR1217345}, p. 460), there exist constants $c_{MW}$ and $\kappa_{MW}$, depending only on $A_0$, such that
\begin{equation}\label{eq:masserwuestholz}
\deg \phi_s \leq c_{MW}[K(s):K]^{\kappa_{MW}},
\end{equation}
independently of $s$. Note that $\mathcal{A}_s$ and $A_0$ are both defined over $K(s)$.

\section{Height bounds for isogenies}\label{sec:heightbounds}
In the previous section, we took as $\phi_s$ just any isogeny between $A_0$ and $\mathcal{A}_{s}$ of minimal degree. This is fine in the case of elliptic curves, but in arbitrary dimension, we have to pick the distinguished isogeny more carefully. This will be achieved in Proposition \ref{prop:isogenyheightbounds} and Corollary \ref{cor:choiceofisogeny}, where we replace $\phi_s$ by $\phi_s \circ \sigma$ for some well-chosen automorphism $\sigma$ of $A_0$.

Proposition \ref{prop:isogenyheightbounds}(ii) and Corollary \ref{cor:choiceofisogeny}(ii) are essentially contained in Orr's work \cite{MR3377393}, albeit formulated rather differently, and our proofs of these results basically run along the same lines as his. Another way to get the desired bounds on quantities associated to an isogeny between $A_0$ and $\mathcal{A}_s$ ($s \in \mathcal{S}$) would be to replace the use of Orr's Proposition 4.2 from \cite{MR3377393} with the endomorphism estimate from Lemma 5.1 of Masser and W\"{u}stholz in \cite{MR1269495} for $A_0 \times \mathcal{A}_s$ (an improved, completely explicit bound can be deduced from Section 9 of \cite{MR3263028}, Lemme 2.11 of \cite{R18} and Minkowski's second theorem) and an argument as in Section 6 of \cite{MR1269495}. Afterwards, one could continue as we do here and obtain bounds that are polynomial (in the sense of \eqref{eq:preceqdefn}) not necessarily in the degree of the isogeny, but certainly in $[K(s):K]$.

Before we can prove the proposition, we need the following technical lemma.

\begin{lem}\label{lem:symplecticsiegelreduction}
Let $g$ be a natural number and $M \in \M_{2g}(\mathbb{Z})$ with $\det M \neq 0$. Let
\[\mathcal{H} = H\left(M^t \begin{pmatrix} 0 & E_g \\ -E_g & 0\end{pmatrix} M\right).\]
Then there are constants $C=C(g)$ and $\kappa=\kappa(g)$ and matrices $S \in \Sp_{2g}(\mathbb{Z})$, $P \in \M_{2g}(\mathbb{Z})$ such that $M = SP$ and $H(P) \leq C\mathcal{H}^{\kappa}$.
\end{lem}

\begin{proof}
Using elementary row operations from $\GL_{2g}(\mathbb{Z})$, we can write $M = M_1 P_1$ with $M_1 \in \GL_{2g}(\mathbb{Z})$ and $P_1 \in \M_{2g}(\mathbb{Z})$ upper triangular. The (non-zero) diagonal entries of $P_1$ are then bounded by $|\det M|$ and after more row operations we can assume that the entries above the diagonal entry $d$ lie in the set $\{0,1,\hdots,|d|-1\}$. So we can assume without loss of generality that $H(P_1)$ is bounded by $|\det M|$, which is of course polynomially bounded in $\mathcal{H}$. Then
\[\mathcal{H}' = H\left(M_1^t \begin{pmatrix} 0 & E_g \\ -E_g & 0\end{pmatrix} M_1\right)\]
is also polynomially bounded in $\mathcal{H}$, so it suffices to prove the lemma for $M_1$ and $\mathcal{H}'$ instead of $M$ and $\mathcal{H}$.

The lemma is now a consequence of Orr's Lemma 4.3 in \cite{MR3377393}, which can be reformulated as asserting that there exists $P_2 \in \GL_{2g}(\mathbb{Z})$ of height bounded polynomially in $\mathcal{H}'$ such that $M_1 P_2 \in \Sp_{2g}(\mathbb{Z})$.
\end{proof}

Before we can state the next theorem, we have to define what a Siegel fundamental domain for the action of (a finite-index subgroup of) $\Sp_{2g}(\mathbb{Z})$ on $\mathbb{H}_g$ is. We give the definition that goes back to Siegel in \cite{MR0001251}, \S 2.

\begin{defn}\label{defn:siegelfundamentalset}
\begin{enumerate}
\item A positive definite symmetric matrix
\[M = (m_{ij})_{i,j=1,\hdots,g} \in \M_g(\mathbb{R})\]
is called Minkowski-reduced if $v^tMv \geq m_{ii}$ for all $v^t = (v_1,\hdots,v_g) \in \mathbb{Z}^g$ with $\gcd(v_i,\hdots,v_g) = 1$ and all $i = 1, \hdots, g$ and $m_{i,i+1} \geq 0$ for all $i = 1,\hdots, g-1$.
\item The Siegel fundamental domain for $\Sp_{2g}(\mathbb{Z})$ or the Siegel fundamental domain is the set of $\tau = (\tau_{ij})_{i,j=1,\hdots,g} \in \mathbb{H}_g$ such that $\det(\Im(M[\tau])) \leq \det(\Im \tau)$ for all $M \in \Sp_{2g}(\mathbb{Z})$, $\Im \tau$ is Minkowski-reduced and $|\Re \tau_{ij}| \leq \frac{1}{2}$ ($i,j = 1, \hdots,g)$.
\item If $F$ denotes the Siegel fundamental domain for $\Sp_{2g}(\mathbb{Z})$, $G \subset \Sp_{2g}(\mathbb{Z})$ is a subgroup of finite index and $g_1=E_{2g},g_2, \hdots, g_n$ is a system of representatives for its right cosets, then $\bigcup_{j=1}^{n}{g_jF}$ is called a Siegel fundamental domain for $G$.
\end{enumerate}
\end{defn}

It is a classical fact that for only finitely many $M \in \Sp_{2g}(\mathbb{Z})$ there exists some $\tau$ in the Siegel fundamental domain with $M[\tau]$ also in the Siegel fundamental domain and that every element of $\mathbb{H}_g$ can be brought into the Siegel fundamental domain by some element of $\Sp_{2g}(\mathbb{Z})$. The same facts then easily follow for every Siegel fundamental domain for some subgroup of $\Sp_{2g}(\mathbb{Z})$ of finite index. This is everything we will need to know about Siegel fundamental domains in this section.

\begin{prop}\label{prop:isogenyheightbounds}
Let $A$ and $B$ be two abelian varieties of dimension $g$, defined over $\mathbb{C}$ and uniformized as $\mathbb{C}^g/\Omega_A\mathbb{Z}^{2g}$ and $\mathbb{C}^g/\Omega_B\mathbb{Z}^{2g}$ respectively, where $\Omega_{A}=(T_{A}~~E_g)$ and $\Omega_{B}=(T_B~~E_g)$ with $T_A$, $T_B \in F$ and $F$ denotes a Siegel fundamental domain for a subgroup of $\Sp_{2g}(\mathbb{Z})$ of finite index. Let $\mathcal{M}$ and $\mathcal{N}$ be ample line bundles on $A$ and $B$ respectively which are associated to the Hermitian forms given by $(\Im T_A)^{-1}$ and $(\Im T_B)^{-1}$ respectively on $\mathbb{C}^g$.

Let $\phi: A \to B$ be an isogeny. Then there exist constants $C$ and $\kappa$, depending only on $F$, $A$, $\Omega_{A}$ and $\mathcal{M}$, but not on $B$ or $\phi$, a natural number $n \in \mathbb{N}$, an automorphism $\sigma: AÊ\to A$ and a matrix $\PhiÊ\in \M_{2g}(\mathbb{Z})$ such that
\begin{enumerate}[label=(\roman*)]
\item $((\phi \circÊ\sigma)^{\ast}\mathcal{N})^{\otimes n}Ê\otimes \mathcal{M}^{\otimes(-1)}$ is ample and $n \leq C(\deg \phi)^{\kappa}$.
\item $\Phi$ is the rational representation of $\phi \circÊ\sigma$ with respect to the lattice bases given by $\Omega_{A}$ and $\Omega_{B}$ and $H(\Phi) \leq C(\degÊ\phi)^{\kappa}$.
\end{enumerate}
\end{prop}

\begin{proof}
Let $\phi_{\mathcal{M}}$ and $\phi_{\mathcal{N}}$ be the principal polarizations induced by $\mathcal{M}$ and $\mathcal{N}$ respectively. Consider $\psi = \phi_{\mathcal{M}}^{-1} \circ \hat{\phi} \circ \phi_{\mathcal{N}} \circÊ\phi \in \End(A)$. It is symmetric, i.e. $\psi' = \psi$, where $\psi' = \phi_{\mathcal{M}}^{-1} \circÊ\hat{\psi} \circ \phi_{\mathcal{M}}$ denotes the Rosati involution. It is also totally positive (or positive definite in the terminology of \cite{MR3377393}) by Theorem 5.2.4 of \cite{MR2062673}, since $\hat{\phi} \circ \phi_{\mathcal{N}} \circÊ\phi = \phi_{\phi^{\ast}\mathcal{N}}$ is a polarization of $A$.

Therefore, we can apply Orr's Proposition 4.2 in \cite{MR3377393} and deduce that there is a constant $c$, depending only on $A$ and $\Omega_{A}$, and $\sigma \in \Aut(A)$ such that the rational representation of $\sigma' \circ \psi \circ \sigma$ with respect to the lattice given by $\Omega_{A}$ has height bounded by $c(\deg \phi)^2$ (we choose $(\End A,')$ as $(R,\dagger)$ and the rational representation with respect to the lattice given by $\Omega_{A}$ as $\rho$). We have $\sigma' \circ \psi \circ \sigma = \phi_{\mathcal{M}}^{-1} \circ \widehat{(\phi \circÊ\sigma)} \circ \phi_{\mathcal{N}} \circÊ(\phi \circÊ\sigma)$, so we can replace $\phi$ by $\phi \circ \sigma$ and $\psi$ by $\sigma' \circ \psi \circÊ\sigma$ and verify (i) and (ii) for this new $\phi$ (and $\sigma=\id$), where $\Phi \in \M_{2g}(\mathbb{Z})$ is the rational representation of $\phi$ with respect to the lattice bases given by $\Omega_{A}$ and $\Omega_{B}$. We have $|\det \Phi| = \deg \phi \neq 0$.

Let $H_{\mathcal{M}}$ and $H_{\mathcal{N}}$ be the Hermitian forms on $\mathbb{C}^g$ associated to $\mathcal{M}$ and $\mathcal{N}$ respectively and let $A'$ and $B'$ be the matrices in $\M_{2g}(\mathbb{R})$ that represent the symmetric positive definite forms $\Re H_{\mathcal{M}}$ and $\Re H_{\mathcal{N}}$ with respect to the lattice bases given by $\Omega_{A}$ and $\Omega_{B}$ respectively. Let $M_1 \in \M_{2g}(\mathbb{Z})$ be the rational representation of $\psi$ with respect to the lattice basis given by $\Omega_{A}$. Now $\psi$ satisfies $\phi_{\mathcal{M}} \circ \psi = \hat{\phi} \circ \phi_{\mathcal{N}} \circÊ\phi$. By taking the analytic representations of both sides, where the dual abelian varieties are canonically uniformized as quotients of the vector space of $\mathbb{C}$-antilinear maps from $\mathbb{C}^g$ to $\mathbb{C}$, it follows (with Lemma 2.4.5 from \cite{MR2062673}) that
\[ÊH_{\mathcal{M}}(\psi(v),w)=H_{\mathcal{N}}(\phi(v),\phi(w))\]
for all $v,w \in \mathbb{C}^g$, where we use $\phi$ and $\psi$ also for the linear maps from $\mathbb{C}^g$ to $\mathbb{C}^g$ corresponding to the analytic representations of $\phi$ and $\psi$ with respect to the given uniformization. By taking real parts and passing to rational representations, we deduce that
\[Ê(M_1v)^t A'w = v^t \Phi^t B' \Phi w\]
for all $v,w \in \mathbb{R}^{2g}$ and it follows that $M_1^t=\Phi^tB'\Phi(A')^{-1}$ and therefore $M_1 = (A')^{-1}\Phi^tB'\Phi$.

Let $H_{\phi^{\ast}\mathcal{N}}$ be the Hermitian form associated to $\phi^{\ast}\mathcal{N}$. The ampleness of $(\phi^{\ast}\mathcal{N})^{\otimes n}Ê\otimes \mathcal{M}^{\otimes(-1)}$ is equivalent to the positive definiteness of its Hermitian form $H_n = nH_{\phi^{\ast}\mathcal{N}}-H_{\mathcal{M}}$ and this is equivalent to the positive definiteness of the symmetric bilinear form $\Re H_n$. One computes that $\Re H_n$ is represented by $M_2=n\Phi^{t}B'\Phi-A'$ with respect to the lattice given by $\Omega_{A}$. Let $v \in \mathbb{R}^{2g}$ be an arbitrary non-zero vector and $M_3=\Phi^t B' \Phi$. Then we have
\[v^{t}M_2v =nv^{t}M_3v-v^{t} M_3 (M_1^{-1}v),\]
and using the Cauchy-Schwarz inequality for the scalar product given by $M_3$ we obtain
\[ v^{t}M_2v  \geq \sqrt{v^{t} M_3 v}\left(n\sqrt{v^{t} M_3 v}-\sqrt{(M_1^{-1}v)^{t}M_3(M_1^{-1}v)}\right).\]

In order to make this quantity positive, $n$ must be bigger than the operator norm of $M_1^{-1}$ with respect to the scalar product given by $M_3$, i.e.
\[ n > \sqrt{(M_1^{-1}v_0)^{t}M_3(M_1^{-1}v_0)}\]
for every $v_0 \in \mathbb{R}^{2g}$ with $v_0^{t}M_3v_0=1$.

We know from Orr's proposition that all coefficients of $M_1$ are bounded by $c(\deg \phi)^2$. Therefore, we can bound the coefficients of both $M_3 = A'M_1$ and $M_3^{-1} = M_1^{-1}(A')^{-1}$ by some power of $\deg \phi$ times a constant, where the constant depends only on $A$, $\Omega_{A}$ and $\mathcal{M}$ (note that $|\det M_1| = (\deg \phi)^2 \geq 1$, so we have a similar bound for the coefficients of $M_1^{-1}$ as for the coefficients of $M_1$).

Since $B'$ and hence $M_3$ is symmetric and positive definite, there is a matrix $\tilde{M_3} \in \GL_{2g}(\mathbb{R})$ such that $M_3=\tilde{M_3}^{t}\tilde{M_3}$. We can then write  $M_3^{-1} = \tilde{M_3}^{-1}\tilde{M_3}^{-t}$, so the coefficients of $\tilde{M_3}$ and $\tilde{M_3}^{-1}$ must be similarly bounded.

When $(\tilde{M_3} v_0)^{t}\tilde{M_3} v_0=1$, the coordinates of $\tilde{M_3} v_0$ are at most $1$ in absolute value. Hence, those of $v_0=\tilde{M_3}^{-1}(\tilde{M_3} v_0)$ are also bounded by some power of $\deg \phi$ times a constant which depends only on $A$, $\Omega_{A}$ and $\mathcal{M}$. Finally we fix $n$ to be the largest integer with
\[ n \leq \sqrt{(M_1^{-1}v_0)^{t}M_3(M_1^{-1}v_0)}+1\]
and obtain a bound of the desired form. This proves (i).

For (ii), we have $\Phi^t[T_{B}]=T_{A}$ for the partial action of $GL_{2g}(\mathbb{Q})$ on $\mathbb{H}_g$ that restricts to the usual action of $\Sp_{2g}(\mathbb{Z})$.

By Lemma \ref{lem:symplecticsiegelreduction}, we can write $\Phi=SP$, where $S\in \Sp_{2g}(\mathbb{Z})$ and $P \in \M_{2g}(\mathbb{Z})$ with $H(P)$ bounded polynomially in $H\left(\Phi^t \left(\begin{smallmatrix}0 & E_g \\ -E_g & 0 \end{smallmatrix}\right)\Phi\right)$. But now $\Phi^t \left(\begin{smallmatrix}0 & E_g \\ -E_g & 0 \end{smallmatrix}\right) \Phi$ represents the imaginary part of the Hermitian form $H_{\phi^{\ast}\mathcal{N}}$ with respect to the lattice basis given by $\Omega_{A}$ (here we use that the lattice basis associated to $\Omega_{B}$ is symplectic with respect to $H_{\mathcal{N}}$). We have
\[ |\Im H_{\phi^{\ast}\mathcal{N}}(v,w)|^2 \leq |H_{\phi^{\ast}\mathcal{N}}(v,w)|^2 \leq H_{\phi^{\ast}\mathcal{N}}(v,v)H_{\phi^{\ast}\mathcal{N}}(w,w)\]
by Cauchy-Schwarz, where $v, w \in \mathbb{C}^g$.

Furthermore, we know that
\[H_{\phi^{\ast}\mathcal{N}}(v,v)H_{\phi^{\ast}\mathcal{N}}(w,w) = \Re H_{\phi^{\ast}\mathcal{N}}(v,v)\Re H_{\phi^{\ast}\mathcal{N}}(w,w).\]
But $\Re H_{\phi^{\ast}\mathcal{N}}$ is represented by $M_3 = \Phi^t B' \Phi$ and we have already bounded the coefficients of that matrix. So the coefficients of $\Phi^t \left(\begin{smallmatrix}0 & E_g \\ -E_g & 0 \end{smallmatrix}\right) \Phi$ are also bounded polynomially in $\deg \phi$ and as they are integers, their height is similarly bounded.

This means we have written $\Phi=SP$, where $S\in \Sp_{2g}(\mathbb{Z})$ and $P \in \M_{2g}(\mathbb{Z})$ with $H(P)$ polynomially bounded in $\deg \phi$. Furthermore, $S^t[T_{B}]$ is an element of $\mathbb{H}_g$. There is $R \in \Sp_{2g}(\mathbb{Z})$ such that $(RS^t)[T_{B}]$ lies again in the Siegel fundamental domain. By \cite{MR3020307}, Lemma 3.2, the height of $R$ is polynomially bounded in terms of the maximum of the absolute values of the coefficients of $S^t[T_{B}]$ together with $1$ and $(\det \Im S^t[T_{B}])^{-1}$. Note that such a bound holds for the Siegel fundamental domain as defined here although in \cite{MR3020307} Siegel's definition from \cite{MR0164063} is used, which demands that $\left(\Im\tau\right)^{-1}$ instead of $\Im \tau$ is Minkowski-reduced, since by Lemma 3.3 of \cite{MR3020307} and Lemma 3.1(3) of \cite{MR3152943} one can switch between the two fundamental domains in a (polynomially) controlled way.

In order to bound the absolute values of the coefficients of $S^t[T_{B}]$ as well as $(\det \Im S^t[T_{B}])^{-1}$, we consider the matrix $M_4 = S^t B' S = P^{-t}M_3P^{-1}$. Recall that $\det P = \det \Phi \neq 0$. As we have a bound on the coefficients of $M_3$ and on $H(P)$, we deduce a similar bound for the coefficients of $M_4$. If we write $S = \left(\begin{smallmatrix} S_1^t & S_3^t \\ S_2^t & S_4^t \end{smallmatrix}\right)$, then we see that $M_4$ represents the real part of $H_{\mathcal{N}}$ with respect to the lattice basis given by the columns of $\left(\begin{smallmatrix} T_B S_1^t+S_2^t & T_B S_3^t + S_4^t\end{smallmatrix}\right)$.

In order to compute $M_4$, it is useful to choose the basis given by the columns of $T_B S_3^t + S_4^t$ for $\mathbb{C}^g$. That this matrix has non-zero determinant (and hence its columns form a basis) follows from the proof that $\Sp_{2g}(\mathbb{Z})$ acts on $\mathbb{H}_g$ by $(U,\tau) \mapsto U[\tau]$.

With respect to this new basis of $\mathbb{C}^g$, the lattice basis given by the columns of $\left(\begin{smallmatrix} T_B S_1^t+S_2^t & T_B S_3^t + S_4^t\end{smallmatrix}\right)$ is given by the matrix $\left(\begin{smallmatrix} S^t[T_B] & E_g\end{smallmatrix}\right)$. Furthermore, the Hermitian form $H_{\mathcal{N}}$ is given by $(S_3T_B+S_4)(\Im T_B)^{-1}(\overline{T_B}S_3^t+S_4^t) = (\Im S^t[T_B])^{-1}$ with respect to this new basis of $\mathbb{C}^g$ (see the calculation in \cite{MR2062673}, p. 214).

With this new basis for $\mathbb{C}^g$, it is easy to compute
\[ M_4 = \begin{pmatrix} M_5 &  (\Re S^t[T_B])(\Im S^t[T_B])^{-1}\\ (\Im S^t[T_B])^{-1}(\Re S^t[T_B]) & (\Im S^t[T_B])^{-1}\end{pmatrix},\]
where
\[M_5 = (\Re S^t[T_B])(\Im S^t[T_B])^{-1}(\Re S^t[T_B])+\Im S^t[T_B].\]
Here, we used that $S^t[T_B]$ and hence both its real and imaginary part are symmetric.

Now our bound on the coefficients of $M_4$ yields first an upper bound on $(\det \Im S^t[T_B])^{-1}$ and on the coefficients of $M_5$. Next, we deduce $\det \Im S^t[T_B] \leq \det M_5$ from Minkowski's determinant inequality (see \cite{MR1215484}, Chapter II, Theorem 4.1.8), since both $\Im S^t[T_B]$ and $M_5 - \Im S^t[T_B]$ are symmetric and positive semidefinite. From this follows an upper bound for $\det \Im S^t[T_B]$. Together with our bound on the coefficients of $M_4$, this readily gives a bound for the coefficients of $\Im S^t[T_B]$ and $\Re S^t[T_B]$ and thereby a bound for the coefficients of $S^t[T_B]$ in absolute value. Thus, we can apply Lemma 3.2 of \cite{MR3020307} to bound $H(R)$ in the required way.

We note that $T_{B}$ lies in the Siegel fundamental domain of a finite-index subgroup of $\Sp_{2g}(\mathbb{Z})$ and $(RS^t)[T_{B}]$ lies in the Siegel fundamental domain of $\Sp_{2g}(\mathbb{Z})$ itself. Therefore, $RS^t$ has to lie in a certain finite set which depends only on $F$ and $g$. Thus, we obtain a similar bound for $H(S)=H(R^{-1}RS^t)$ and thereby for $H(\Phi)=H(SP)$, since we have already bounded $H(P)$.
\end{proof}

In order to state the next corollary, we introduce the following notation that will also be used in the following sections: We write $f \preceq g$ for (positive) quantities $f$ and $g$, if there exist constants $c>0$ and $\kappa>0$, depending on $K$, $A_0$, $L_0$, $\tau_0$, $\Gamma$, $l$, $\mathcal{A}$, $\mathcal{L}$, $\mathcal{S}$, $\mathcal{C}$ and the choice of a Siegel fundamental domain for $G(l,2l)$ such that
\begin{equation}\label{eq:preceqdefn}
f \leq c\max\{1,g\}^{\kappa}.
\end{equation}
The choice of a Siegel fundamental domain for $G(l,2l)$ will be made implicitly in Proposition \ref{prop:uniformization}.

\begin{cor}\label{cor:choiceofisogeny}
Let $s \in \mathcal{S}$ such that $A_0$ and $\mathcal{A}_s$ are isogenous. Choose $\tau$ in a Siegel fundamental domain for the action of $G(l,2l)$ on $\mathbb{H}_g$ such that $\iota(\tau) = s$ with $\iota$ as in Proposition \ref{prop:universalfamily}. Then there exist an isogeny $\phi_s: A_0 \to \mathcal{A}_s$ of minimal degree (as defined before \eqref{eq:masserwuestholz}), a natural number $M \in \mathbb{N}$ and a matrix $\Phi \in \M_{2g}(\mathbb{Z})$ such that
\begin{enumerate}[label=(\roman*)]
\item $(\phi_s^{\ast}\mathcal{L}_s)^{\otimes M}Ê\otimes L_0^{\otimes(-1)}$ is ample and $M \preceq \deg \phi_s$.
\item $\Phi$ is the rational representation of $\phi_s$ with respect to the uniformizations $\exp_0$ and $\exp(\tau,\cdot)$ and the lattice bases given by $\Omega_\tau$ and $\Omega_{\tau_0}$, where $\exp$, $\exp_0$, $\Omega_\tau$ and $\Omega_{\tau_0}$ are defined as in Section \ref{sec:preliminaries}. It satisfies $H(\Phi) \preceq \deg \phi_s$.
\end{enumerate}
\end{cor}

\begin{proof}
Let $\phi$ be any isogeny of minimal degree between $A_0$ and $\mathcal{A}_s$. We apply Proposition \ref{prop:isogenyheightbounds} to $\phi$ with $A=A_0$, $B=\mathcal{A}_s$, $T_A = \tau_0$, $T_B = \tau$, $\mathcal{M}=L_0$ and $\mathcal{N}=\mathcal{L}_s'$, where $\mathcal{L}_s = (\mathcal{L}_s')^{\otimes l}$ by Proposition \ref{prop:universalfamily}(vi). Putting $\phi_s = \phi \circ \sigma$ yields what we want: Since $\sigma$ is an automorphism, we have $\deg \phi_s = \deg \phi$, so $\phi_s$ is of minimal degree. As $(\phi_s^{\ast}\mathcal{L}_s')^{\otimes n}Ê\otimes L_0^{\otimes(-1)}$ is ample, so is $(\phi_s^{\ast}\mathcal{L}_s')^{\otimes ln}Ê\otimes L_0^{\otimes(-l)} \otimes L_0^{\otimes(l-1)} = (\phi_s^{\ast}\mathcal{L}_s)^{\otimes n}Ê\otimes L_0^{\otimes(-1)}$ and thus we may take $M = n$. Note that $G(l,2l)$ has finite index in $\Sp_{2g}(\mathbb{Z})$ and that $\tau_0$ was already chosen in the Siegel fundamental domain for $\Sp_{2g}(\mathbb{Z})$. The implicit constants depend only on $A_0$, $L_0$, $\tau_0$ and the chosen Siegel fundamental domain, but are independent of $s$ and $\tau$.
\end{proof}

Finally, we record a lemma due to R\'{e}mond that allows us to bound the height of a basis of the lattice corresponding to an abelian subvariety of $A_0$ in terms of the degree of the abelian subvariety.
\begin{lem}\label{lem:degreeabeliansubvariety}
Let $B_0$ be an abelian subvariety of $A_0$ of codimension $k$ and denote by $\deg B_0$ its degree with respect to the ample line bundle $L_0$. Under the identification of $\mathbb{R}^{2g}$ with $\mathbb{C}^g$ given by $u \mapsto \Omega_{\tau_0}u$, there exists a matrix $HÊ\in \M_{2g \times 2(g-k)}(\mathbb{Z})$ such that $\exp_0^{-1}(B_0(\mathbb{C})) = \{Hy+z ; y \in \mathbb{R}^{2(g-k)}, z \in \mathbb{Z}^{2g}\}$, $\Omega_{\tau_0}H$ has rank equal to $g-k$ and $\lVert H \rVert \preceq \deg B_0$. Here, $\exp_0$ and $\Omega_{\tau_0}$ are defined as in Section \ref{sec:preliminaries}.
\end{lem}

\begin{proof}
We follow R\'{e}mond's construction in Section 4 of \cite{MR1804159}. We obtain a basis $w_i = \sum_{j=1}^{2g}{\lambda^{(i)}_j v_j}$ ($i=1,\hdots,2(g-k)$) of the connected component of $\exp_0^{-1}(B_0(\mathbb{C}))$ containing $0$ under the given identification of $\mathbb{R}^{2g}$ and $\mathbb{C}^g$. Here, $v_1,\hdots,v_{2g}$ is a suitable basis of $\mathbb{Z}^{2g}$ that is chosen depending on $L_0$, but independently of $B_0$, and the $\lambda^{(i)}_j$ are integers.

By an inequality on p. 531 of \cite{MR1804159}, we have
\[Ê\lVert w_i \rVert \lVert w_{2(g-k)+1-i}\rVert \preceq \deg B_0 \quad (i=1,\hdots,2(g-k)),\]
where $\lVert \cdot \rVert$ is a Euclidean norm on $\mathbb{R}^{2g}$ induced by $L_0$. This norm is bounded from below on $\mathbb{Z}^{2g}\backslash\{0\}$ by a positive constant that doesn't depend on $B_0$, which implies that
\[Ê\lVert w_i \rVert \preceq \deg B_0 \quad (i=1,\hdots,2(g-k)).\]

Since all norms on finite-dimensional real vector spaces are equivalent and $\lVert \cdot \rVert$ doesn't depend on $B_0$, it follows that $|\lambda_j^{(i)}| \preceq \deg B_0$ ($i=1,\hdots,2(g-k)$, $j = 1,\hdots,2g$). We deduce that the coordinates of $w_1, \hdots, w_{2(g-k)}$ with respect to the basis $v_1, \hdots, v_{2g}$ of $\mathbb{Z}^{2g}$ (which is not necessarily the standard one) are bounded. However, this basis is chosen independently of $B_0$ and so we obtain a comparable bound for the coordinates with respect to the standard basis.

We now take as $H$ the matrix with columns $w_1, \hdots, w_{2(g-k)}$. The columns of the matrix $\Omega_{\tau_0}H$ span the connected component of $\exp_0^{-1}(B_0(\mathbb{C}))$ containing $0$ seen as a $(g-k)$-dimensional vector subspace of $\mathbb{C}^g$ and so this matrix has rank equal to $g-k$.
\end{proof}

\section{Galois orbit bounds}\label{sec:galoisorbitbounds}

In this section, we show that virtually all occurring important quantities can be bounded polynomially in terms of $[K(p):K]$, where $p$ is a point in $\mathcal{A}^{[k]}_{\Gamma}\cap \mathcal{C}$ (reversing the direction of the inequalities leads to lower bounds for $[K(p):K]$ in terms of these other quantities -- hence the title ``Galois orbit bounds"). We will need two lemmata before we can prove the crucial Proposition \ref{prop:galoisbounds}. From now on, we will always take the isogeny given by Corollary \ref{cor:choiceofisogeny} as $\phi_s$. There might be some ambiguity in the choice of $\tau$ if it lies on the boundary of the Siegel fundamental domain for $G(l,2l)$, but this ambiguity doesn't change the construction in Proposition \ref{prop:isogenyheightbounds} -- which only depends on the principal polarization induced by $\mathcal{L}_s'$ and the data associated to $A_0$ -- and hence has no influence on $\phi_s$. Likewise, the implicit constants in the estimates are the same for any choice of $\tau$ in the Siegel fundamental domain.

\begin{lem}\label{lem:heightbound}
Let $s \in \mathcal{S}$ be such that $\mathcal{A}_s$ and $A_0$ are isogenous. Then there are constants $c_1$ and $c_2$, depending on $K$ and $A_0$, but independent of $s$ such that
\[ h_{\overline{\mathcal{S}}}(s) \leq c_1\log[K(s):K]+c_2.\]
\end{lem}

\begin{proof}
We will use $c_1, c_2, \hdots$ for constants depending on $K$ and $A_0$, but independent of $s$. We will denote the stable Faltings height of $\mathcal{A}_s$ as defined in \cite{MR718935} by $h_F(\mathcal{A}_s)$.

By Faltings' Lemma 5 in \cite{MR718935}, we have
\begin{equation}\label{eq:faltingsestimate}
h_F(\mathcal{A}_s) \leq  h_F(A_0)+\frac{\log \deg \phi_s}{2}.
\end{equation}
By an inequality of Bost-David (Pazuki's Corollary 1.3 (1) in \cite{MR2903770}), we know that
\[\left| h_{\overline{\mathcal{S}}}(s) - \frac{1}{2}  h_F(\mathcal{A}_s)  \right| \leq c_3\log(\max\{h_{\overline{\mathcal{S}}}(s),1\})+c_4\]
for some constants $c_3$ and $c_4$, depending only on $g$ and $l$. Our choice of embedding of $A_{g,l}$ and $\mathfrak{A}_{g,l}$ into projective space through the use of Theta functions means that our $h_{\overline{\mathcal{S}}}(s)$ differs from the Theta height of $\mathcal{A}_s$ in Pazuki's work with $l=r^2$ only by an amount that is bounded independently of $s$: Pazuki uses another norm at the archimedean places for the definition of his height and he uses another coordinate system as he notes after his Definition 2.6, but by \cite{MR0325625}, p. 171, this coordinate system is related to ours by an invertible linear transformation with algebraic coefficients.

We deduce that
\begin{equation}\label{eq:j-faltings}
h_{\overline{\mathcal{S}}}(s) \leq c_5\max\{h_{F}(\mathcal{A}_s),1\}.
\end{equation}
Combining \eqref{eq:masserwuestholz}, \eqref{eq:faltingsestimate} and \eqref{eq:j-faltings}, we obtain that
\[ h_{\overline{\mathcal{S}}}(s) \leq c_1\log[K(s):K]+c_2\]
for some constants $c_1$ and $c_2$.
\end{proof}

\begin{lem}\label{lem:curveheight}
Let $p \in \mathcal{C}$ with $s = \pi(p) \in \mathcal{S}$ and suppose that $\mathcal{C}$ is not contained in $\mathcal{A}_s$. Then we have $\widehat{h}_s(p) \preceq h_{\overline{\mathcal{S}}}(s)$.
\end{lem}

Our proof even yields a bound that is linear in $h_{\overline{\mathcal{S}}}(s)$, but a polynomial bound will suffice for our purposes. We note that it is crucial for this lemma that $\mathcal{C}$ is a curve and not a subvariety of $\mathcal{A}$ of higher dimension. Indeed, the main obstacle that one encounters attempting to generalize Theorem \ref{thm:main} to higher-dimensional subvarieties $\mathcal{V} \subset \mathcal{A}$ which dominate the base is the lack of such a height bound for (a large enough subset of) the points in $\mathcal{A}_\Gamma \cap \mathcal{V}$.

\begin{proof}
We use $c_6, \hdots$ for constants that depend only on $\mathcal{A}$ and $\mathcal{C}$. Let for the moment $s \in \mathcal{S}$ and $p \in \mathcal{A}_s$ be arbitrary. We will first bound $\widehat{h}_s(p)$ in terms of $h_s(p)$ and $h_{\overline{\mathcal{S}}}(s)$. It would be possible to use Silverman's Theorem A in \cite{MR703488} for this; there is however the problem that $\overline{\mathfrak{A}_{g,l}}$ and $\overline{A_{g,l}}$ are usually not smooth, so one would either need to construct a more sophisticated (i.e. smooth) compactification of the universal family (this was achieved by Pink in his dissertation \cite{MR1128753}) or adapt Silverman's proof by using Cartier instead of Weil divisors.

Another, more elementary way is to use Lemma 3.4 of \cite{MR1207211}. It is shown in that lemma that there exists a family of polynomials $P_{i,j}$ ($i=0,\hdots,l^g-1$, $j=1,\hdots,J$)  in the projective coordinates of $s \in \mathcal{S}$ and $p \in \mathcal{A}_s \subset \mathbb{P}^{l^g-1}$ with the following properties: Every $P_{i,j}$ is a polynomial with integer coefficients, homogeneous of degree $2(l^{8g}-1)$ in the coordinates of $s$ and homogeneous of degree $4$ in the coordinates of $p$. For every $s \in \mathcal{S}$ and $p \in \mathcal{A}_s$ and every $j \in \{1,\hdots,J\}$, the $P_{i,j}(s,p)$ ($i=0,\hdots,l^g-1$) are either all zero or they are the projective coordinates of $2p$ in $\mathcal{A}_s \subset \mathbb{P}^{l^g-1}$ (by abuse of notation, $P_{i,j}(s,p)$ denotes $P_{i,j}$ evaluated at the projective coordinates of $s$ and $p$). Furthermore, there exists $j \in \{1,\hdots,J\}$, depending on $s$ and $p$, such that not all $P_{i,j}(s,p)$ ($i=0,\hdots,l^g-1$) are zero.

Fixing $j \in \{1,\hdots,J\}$ and following the proof of Theorem B.2.5(a) in \cite{MR1745599} (which amounts to the triangle inequality), we get a bound of the form
\[ h_s(2p) \leq 4h_s(p) + 2(l^{8g}-1) h_{\overline{\mathcal{S}}}(s)+c_6,\]
where $c_6$ depends only on $l$, $g$ and the (integral) coefficients of the $P_{i,j}$, but is independent of $s$ and $p$. The bound is valid for those $s$ and $p$, where not all $P_{i,j}(s,p)$ ($i=0,\hdots,l^g-1$) are zero. After reiterating the process for every $j \in \{1,\hdots,J\}$ and adjusting the constants if necessary, we can assume that the inequality holds for all $s \in \mathcal{S}$ and $p \in \mathcal{A}_s$. We then obtain easily from $\widehat{h}_s(p)=\lim_{n \to \infty}{\frac{h_s(2^n p)}{4^n}}$ that
\[ \widehat{h}_s(p) \leq h_s(p) + \frac{2(l^{8g}-1) h_{\overline{\mathcal{S}}}(s)+c_6}{3},\]
where we used that $\sum_{n=1}^{\infty}{4^{-n}}=\frac{1}{3}$.

Let now $p$ be a point of $\mathcal{C}$ as in the lemma. In view of the above inequality, it suffices to show that $h_s(p) \preceq h_{\overline{\mathcal{S}}}(s)$. Since $\mathcal{C}$ is irreducible and not contained in $\mathcal{A}_s$, the morphism $\pi|_{\overline{\mathcal{C}}}: \overline{\mathcal{C}} \to \overline{\mathcal{S}}$ is quasi-finite. It is also proper, hence finite. 
Therefore, the pullback $\pi^{\ast} L_{\overline{\mathcal{S}}}$ of the ample line bundle $L_{\overline{\mathcal{S}}}$ is also ample.

On the other hand, the closed immersion $\iota: \overline{\mathcal{C}} \hookrightarrow \overline{\mathfrak{A}_{g,l}}$ yields a very ample line bundle $\iota^{\ast} \mathcal{L}$ on $\overline{\mathcal{C}}$. It follows from the ampleness of $\pi^{\ast} L_{\overline{\mathcal{S}}}$ that there exists some natural number $NÊ\in \mathbb{N}$ such that $\pi^{\ast} L_{\overline{\mathcal{S}}}^{\otimes N} \otimes \iota^{\ast} \mathcal{L}^{\otimes (-1)}$ is ample.

If we choose associated heights $h_{\overline{\mathcal{C}},\iota^{\ast}\mathcal{L}}$ and $h_{\overline{\mathcal{C}},\pi^{\ast} L_{\overline{\mathcal{S}}}}$, it now follows from fundamental properties of the Weil height that
\[Êh_{\overline{\mathcal{C}},\iota^{\ast}\mathcal{L}}(p) \leq Nh_{\overline{\mathcal{C}},\pi^{\ast} L_{\overline{\mathcal{S}}}}(p)+c_7\]
and then by functoriality that
\[Êh_s(p) \leq Nh_{\overline{\mathcal{S}}}(s)+c_8,\]
whence the lemma follows.
\end{proof}

The next proposition bounds all important quantities in terms of $[K(p):K]$ alone, where $p$ is some point in $\mathcal{A}^{[k]}_{\Gamma} \cap \mathcal{C}$.

\begin{prop}\label{prop:galoisbounds}
Let $s \in \mathcal{S}$ be such that $\mathcal{A}_s$ and $A_0$ are isogenous and $p \in \mathcal{C} \cap \phi_s(\Gamma + B_0)$ for some abelian subvariety $B_0$ of $A_0$. Suppose that $\pi(\mathcal{C}) = \mathcal{S}$. Then there exist $\gamma \in \Gamma$, an abelian subvariety $B_1 \subset B_0$ and $b \in B_1$ with the following properties: If we choose $N \in \mathbb{N}$ minimal with $N\gamma=\sum_{i=1}^{r}{a_i\gamma_i} \in \mathbb{Z}\gamma_1+\hdots+\mathbb{Z}\gamma_r$ and if $\deg B_1$ denotes the degree of $B_1$ with respect to the ample line bundle $L_0$, then we have $p = \phi_s(\gamma + b)$ and
\begin{enumerate}[label=(\roman*)]
\item $\deg \phi_s \preceq [K(p):K]$,
\item $\deg B_1 \preceq [K(p):K]$,
\item $\max\{\left|a_1\right|,\hdots,\left|a_r\right|,N\} \preceq [K(p):K].$
\end{enumerate}
\end{prop}
\begin{proof}
Part (i) is just a restatement of \eqref{eq:masserwuestholz}, where we take into account that $[K(p):K] \geq [K(s):K]$. We have $p = \phi_s(q)$ for some $q \in \Gamma + B_0$. It follows from Corollary \ref{cor:choiceofisogeny} that there exists $M \in \mathbb{N}$ such that $(\phi_{s}^{\ast}\mathcal{L}_s)^{\otimes M} \otimes L_0^{\otimes (-1)}$ is ample and $M \preceq \deg \phi_s \preceq [K(p):K]$. Therefore
\[Ê\widehat{h}_{A_0}(q) \leq M\widehat{h}_{\phi_{s}^{\ast}(\mathcal{L}_s)}(q)=M\widehat{h}_s(\phi_s(q)) = M\widehat{h}_s(p),\]
which implies together with Lemma \ref{lem:heightbound} and Lemma \ref{lem:curveheight} that $\widehat{h}_{A_0}(q) \preceq [K(p):K]$.

We note that $q$ is defined over a field extension of $K(p)$ of degree at most $\eta(g)\deg \phi_s$ for a certain function $\eta: \mathbb{N} \to \mathbb{N}$, since $\phi_s$ is defined over a field extension of $K(s) \subset K(p)$ of degree at most $\eta(g)$ by R\'{e}mond's Th\'{e}or\`{e}me 1.2 in \cite{R17} and $q$ has degree at most $\deg \phi_s$ over the compositum of $K(p)$ and the field of definition of $\phi_s$, since all its Galois conjugates over that field lie in $\phi_s^{-1}(p)$ and this fiber has $\deg \phi_s$ elements. Here, R\'{e}mond has obtained the best possible $\eta$, while the fact that the bound depends only on $g$ goes back to Silverberg in \cite{MR1154704} and Masser-W\"ustholz in \cite{MR1207211}, Lemma 2.1. Hence, we have $[K(q):K] \preceq [K(p):K]$.

Consider the point $\tilde{q} = (q,\gamma_1,\hdots,\gamma_r) \in A_0^{r+1}$. Let $B$ be the smallest abelian subvariety of $A_0^{r+1}$ such that a multiple $\mu\tilde{q}$ of $\tilde{q}$ lies inside $B$ ($\muÊ\in \mathbb{N}$). By Proposition 9.1 of \cite{MR3552014}, we have $\deg B \preceq \max\{\widehat{h}_{A_0}(q),[K(q):K]\}Ê\preceq [K(p):K]$. Here, $\deg B$ denotes the degree of $B$ with respect to the line bundle $\pi_1^{\ast} L_0 \otimes \cdots \otimes \pi_{r+1}^{\ast}L_0$, where $\pi_i: A_0^{r+1} \to A_0$ is the projection to the $i$-th factor ($i=1,\hdots,r+1$).

Let $\Omega$ be a finite set of abelian varieties over $\bar{\mathbb{Q}}$ such that every quotient $A_0^{r+1}/H$ for some abelian subvariety $H$ of $A_0^{r+1}$ is isogenous over $\bar{\mathbb{Q}}$ to some element of $\Omega$. For each $A' \in \Omega$ we can fix some norm $\lVertÊ\cdotÊ\rVert_{A'}$ on $\Hom(A_0^{r+1},A') \otimesÊ\mathbb{R}$ and a symmetric ample line bundle on $A'$ to obtain a canonical height $\widehat{h}_{A'}$ on $A'$. After passing to a finite field extension, we can assume that all $A' \in \Omega$, all these line bundles as well as all elements of $\Hom(A_0^{r+1},A')$ for all $A' \in \Omega$ are defined over $K$.

Going through the proof of Proposition 9.1 in \cite{MR3552014}, we see that $B$ is obtained as the irreducible component of $\ker \alpha$ containing the neutral element for a surjective homomorphism $\alpha: A_0^{r+1} \to A$ for some $AÊ\in \Omega$. If we write $\lVert \cdotÊ\rVert = \lVert \cdot \rVert_{A}$, then we even obtain from Lemma 9.5 of \cite{MR3552014} a surjective homomorphism $\alpha: A_0^{r+1} \to A$ such that $B$ is the irreducible component of $\ker \alpha$ containing the neutral element and $\lVert\alpha\rVert \preceq [K(p):K]$.

We have a projection morphism $\psi: B \to A_0^r$ given by omitting the first coordinate. We let $B' = \psi(B) \subset A_0^r$ and let $B_2$ be the connected component of $\kerÊ\psi = B \cap (A_0 \times \{0\}^{r}) \subset B$ containing the neutral element. Since $q \in \Gamma + B_0$, it follows that $B_2 \subset B_0 \times \{0\}^r$. By Poincar\'{e}'s reducibility theorem, there exists an abelian subvariety $B_3 \subset B$ such that the restriction of the natural addition morphism $B_2 \times B_3 \to B$ is an isogeny. It follows that $\psi|_{B_3}: B_3 \to B'$ must be an isogeny. As usual, there exists an isogeny $\chi: B' \to B_3$ such that $\chi \circ \psi|_{B_3}$ is multiplication by $\deg \psi|_{B_3}$ on $B_3$.

Since $\psi|_{B_3}: B_3 \to B'$ is surjective, we can choose $uÊ\in B_3$ such that $\psi(u) = \mu(\gamma_1,\hdots,\gamma_r)$. Applying Poincar\'{e}'s reducibility theorem again, we find an abelian subvariety $B'' \subset A_0^r$ such that the restriction of the natural addition morphism $B' \times B'' \to A_0^r$ is an isogeny. Again, we get an isogeny $\rho: A_0^r \to B'Ê\times B''$ in the other direction such that their composition is multiplication by a scalar. By projecting to the first coordinate, we obtain $\rho': A_0^r \to B'$. Let $w \in A_0^r$ with $\rho(w) = (\mu(\gamma_1,\hdots,\gamma_r),0)$. It follows that $\mu(\gamma_1,\hdots,\gamma_r)$ is some multiple of $w$ and hence $wÊ\in \Gamma^r$. We have $(\deg \psi|_{B_3})u = \chi(\psi(u)) = \chi(\mu(\gamma_1,\hdots,\gamma_r)) = (\chi \circÊ\rho')(w)$. As $\chi \circÊ\rho': A_0^r \to B_3 \hookrightarrow A_0^{r+1}$, $\Gamma$ is stable under $\End(A_0)$ and $w \in \Gamma^r$, we deduce that $u \in \Gamma^{r+1}$.

It follows from $\psi(u) = \mu(\gamma_1,\hdots,\gamma_r)$ that $\mu(q,\gamma_1,\hdots,\gamma_r) \in u + \kerÊ\psi \subset u + (A_0)^{r+1}_{\tors}+B_2$ and by considering only the first coordinate we see that $\mu q \in \pi_1(u) +Ê(A_0)_{\tors} + \pi_1(B_2) \subset \Gamma + \pi_1(B_2)$ and hence $qÊ\in \Gamma + \pi_1(B_2)$. Now, $B_1 = \pi_1(B_2)$ is an abelian subvariety of $A_0$ of degree $\degÊB_1 = \deg B_2$ with respect to $L_0$. Since $B_2$ is an irreducible component of $B \cap (A_0 \times \{0\}^r)$, we know that $\deg B_2 \preceq \deg B$ by Proposition 3.1 of \cite{R18}. We also know that $B_2 = B_1 \timesÊ\{0\}^{r}$ and so $B_1 \subset B_0$, since $B_2 \subset B_0 \times \{0\}^r$. This proves (ii).

Since $\Hom(A_0^{r+1},A)$ is a finitely generated $\mathbb{Z}$-module and the height is quadratic, there exists a constant $c_0$, depending only on the two abelian varieties and the choices of symmetric ample line bundles as well as the choice of the norm, such that $\widehat{h}_{A}(\alpha'(x)) \leq c_0\lVert\alpha'\rVert^2\sum_{i=1}^{r+1}\widehat{h}_{A_0}(x_i)$ for all $\alpha' \in \Hom(A_0^{r+1},A)$ and all $x = (x_1,\hdots,x_{r+1}) \in A_0^{r+1}$. In particular, this bound holds for our $\alpha$ as chosen above.

We apply R\'{e}mond's Lemme 6.1 in \cite{MR2311666} to choose $\gamma' \in \Gamma$ and $b' \in B_1$ such that $q = \gamma' + b'$ and $\widehat{h}_{A_0}(\gamma') \preceq \widehat{h}_{A_0}(q) \preceq [K(p):K]$. Note that we have assumed $\Gamma = \Gamma_{sat}$ in R\'{e}mond's notation and that R\'{e}mond's Lemme also holds for $\epsilon = 0$ as is the case here. Suppose that $m\gamma' = m_1\gamma_1+\hdots+m_r\gamma_r$ with $mÊ\in \mathbb{N}$, $m_1, \hdots, m_rÊ\in \mathbb{Z}$. Since $\widehat{h}_{A_0}(\gamma') \preceq [K(p):K]$, $\widehat{h}_{A_0}$ extends to a norm on $\Gamma \otimes \mathbb{R}$ and all norms on the finite-dimensional $\mathbb{R}$-vector space $\Gamma \otimes \mathbb{R}$ are equivalent, we also have that $\max_{i=1,\hdots,r}\frac{|m_i|}{m} \preceq [K(p):K]$.

For given $N \in \mathbb{N}$, we can find $n \leq N$ and $a_1, \hdots, a_r \in \mathbb{Z}$ such that $\max_{i=1,\hdots,r}\left|a_i-\frac{nm_i}{m}\right|Ê\leq \lfloor N^{\frac{1}{r}}\rfloor^{-1}$. It follows that
\begin{align*}
\widehat{h}_{A}\left(\alpha\left(nq-\sum_{i=1}^{r}{a_i\gamma_i},0,\hdots,0\right)\right) = \\
\widehat{h}_{A}\left(\alpha\left(n\gamma'-\sum_{i=1}^{r}{a_i\gamma_i},0,\hdots,0\right)\right) \leq c_9\lVert\alpha\rVert^2 N^{-\frac{2}{r}},
\end{align*}
where the constant $c_9$ depends only on $A_0$, $L_0$, $A$, the choice of symmetric ample line bundle on $A$ as well as of the norm $\lVert \cdotÊ\rVert$ on $\Hom(A_0^{r+1},A) \otimes \mathbb{R}$ and $\gamma_1,\hdots,\gamma_r$. Since $\alpha\left(nq-\sum_{i=1}^{r}{a_i\gamma_i},0,\hdots,0\right)$ is defined over $K(q)$, it follows from a theorem of Masser (\cite{MR766295}, p. 154) that it must be a torsion point the order $s$ of which is polynomially bounded in $[K(p):K]$ as soon as $N$ exceeds some bound that is polynomial in $[K(p):K]$ (recall that $\lVert\alpha\rVert \preceqÊ[K(p):K]$).

So we may choose $NÊ\preceqÊ[K(p):K]$ and $a_1, \hdots, a_r \in \mathbb{Z}$ such that $(sNq - \sum_{i=1}^{r}{sa_i\gamma_i},0,\hdots,0) \in \kerÊ\alpha$ and therefore $Nq - \sum_{i=1}^{r}{a_i\gamma_i} \in t + B_1$ for a torsion point $t$. Furthermore, the $a_i$ satisfy by construction $|a_i| \leq \frac{N|m_i|}{m}+1Ê\preceqÊ[K(p):K]$. Applying Proposition 9.1 of \cite{MR3552014} again, we see that $t$ can be chosen such that its order is polynomially bounded in $[K(p):K]$. This proves (iii) and thereby the proposition.
\end{proof}

\section{o-Minimality}\label{sec:ominimality}
We give a brief introduction to the theory of o-minimal structures and define all terms which are relevant in our application. We refer to the book of van den Dries (\cite{MR1633348}) for a more thorough treatment of o-minimal structures.

\begin{defn}
A set $A \subset \mathbb{R}^n$ is called semialgebraic if it is a finite union of sets of the form
\[ \{x \in \mathbb{R}^n; f(x)=0\mbox{, }g_1(x)>0, \hdots, g_s(x)>0\},\]
where $f$, $g_1, \hdots, g_s \in \mathbb{R}[X_1,\hdots,X_n]$. A map $f: A \to \mathbb{R}^m$ ($A \subset \mathbb{R}^n$) is called semialgebraic if its graph is semialgebraic.
\end{defn}

\begin{defn}
An o-minimal structure $\mathfrak{S}$ (over $(\mathbb{R},+,-,\cdot,<,0,1)$) is a sequence $\mathfrak{S}=(\mathfrak{S}_n)_{n \in \mathbb{N}}$ such that $\mathfrak{S}_n$ is a subset of the power set of $\mathbb{R}^n$ for all $n \in \mathbb{N}$ and the following conditions are satisfied:
\begin{enumerate}[label=(\roman*)]
\item $A,B \in \mathfrak{S}_n \implies A \cup B$, $\mathbb{R}^n \backslash A \in \mathfrak{S}_n$.
\item $A \in \mathfrak{S}_n \implies \mathbb{R} \times A \in \mathfrak{S}_{n+1}$.
\item $A \in \mathfrak{S}_{n+1} \implies p_{n}(A) \in \mathfrak{S}_n$, where $p_{n}: \mathbb{R}^{n+1} \to \mathbb{R}^n$ is the projection onto the first $n$ factors.
\item All semialgebraic subsets of $\mathbb{R}^n$ are contained in $\mathfrak{S}_n$.
\item The set $\mathfrak{S}_1$ consists precisely of all finite unions of point sets $\{a\}$ ($a \in \mathbb{R}$) and open intervals $(a,b)$ ($a \in \mathbb{R} \cup \{-\infty\}$, $b \in \mathbb{R} \cup \{\infty\}$).
\end{enumerate}
We call the elements of $\bigcup_{n \in \mathbb{N}}\mathfrak{S}_n$ the definable 
sets with respect to $\mathfrak{S}$ or simply the definable sets (if $\mathfrak{S}$ is fixed).
\end{defn}

Since our uniformization map goes to a product of projective spaces, we need to introduce the notion of a definable space. This notion is treated in more detail by van den Dries in Chapter 10 of \cite{MR1633348}. In the following definitions, definability will always mean definability with respect to some fixed o-minimal structure $\mathfrak{S}$.

\begin{defn}
Suppose that $A \subset \mathbb{R}^m$ and $B \subset \mathbb{R}^n$. A map $f: A \to B$ is called definable if its graph
\[\{(x,f(x))\mbox{; }x \in A\}\]
is definable.
\end{defn}

\begin{defn}
A definable space is a set $S = \cup_{i \in I}{U_i}$ with $I$ finite together with bijective maps $f_i: U_i \to U_i'$, where $U_i' \subset \mathbb{R}^{m_i}$ is a definable set, such that for all $i, j$ the set $f_i(U_i \cap U_j)$ is definable and open in $U_i'$ and the map $f_j \circ (f_i^{-1})|_{f_i(U_i \cap U_j)}: f_i(U_i \cap U_j) \to f_j(U_i \cap U_j)$ is definable and continuous.  A set $X \subset S$ is called definable if $f_i(X \cap U_i)$ is definable for every $i \in I$. We call the $f_i$ charts of $S$.
\end{defn}

\begin{defn}
Suppose that $S$ and $T$ are definable spaces. A map $F: S \to T$ is called a morphism if for every chart $f: U \to U'$ of $S$ and every chart $g: V \to V'$ of $T$ the set $f(U \cap F^{-1}(V))$ is definable and open in $U'$ and the map $gÊ\circ F \circ (f^{-1})|_{f(U \cap F^{-1}(V))}$ is continuous and definable.
\end{defn}

It is easily seen that image and pre-image of a definable set under a definable map or a morphism are definable and that the composition of two definable maps or morphisms is again a definable map or a morphism respectively. A definable map is a morphism with respect to the standard global charts of its domain and its range precisely if it is continuous.

By the Seidenberg-Tarski theorem, the semialgebraic sets themselves form an o-minimal structure (the definable maps of which are the semialgebraic maps). For our purposes, this will not be sufficient and we will have to work in the structure $\mathbb{R}_{\an,\exp}$, which contains (among other things) the graph of the exponential function on the real numbers and the graph of the restriction of any analytic function, defined on an open neighbourhood of $[0,1]^n$, to $[0,1]^n$ ($n \in \mathbb{N}$). That this structure is o-minimal and admits analytic cell decomposition is due to van den Dries and Miller (see \cite{MR1264338}).

In order to prove our main theorem, we will need that rational points on definable sets are sparse unless there is a ``reason" for them not to be sparse in the form of a semialgebraic set, contained in the definable set. This is the famous Pila-Wilkie Theorem. We will use a variant by Habegger and Pila, counting ``semirational" points, which is what we will need in the proof.

\begin{thm}\label{thm:habegger-pila} (Habegger-Pila)
Let $Z \subset \mathbb{R}^m \times \mathbb{R}^{n_1} \times \mathbb{R}^{n_2}$ be a definable set and $\epsilon > 0$. Let $\pi_1$, $\pi_2$ and $\pi_3$ be the projections onto $\mathbb{R}^m$, $\mathbb{R}^{n_1}$ and $\mathbb{R}^{n_2}$ respectively. There is a constant $c=c(Z,\epsilon)>0$ with the following property. If $T \geq 1$ and
\[|\pi_3(\{(y,z_1,z_2) \in Z\mbox{; } y=(y_1,\hdots,y_m) \in \mathbb{Q}^m\mbox{, } \max_{j=1,\hdots,m}H(y_j) \leq T\})| > cT^{\epsilon},\]
there exists a continuous and definable function $\delta: [0,1] \to Z$ such that the following properties hold.
\begin{enumerate}[label=(\roman*)]
\item The composition $\pi_1 \circ \delta: [0,1] \to \mathbb{R}^m$ is semialgebraic and its restriction to $(0,1)$ is real analytic.
\item The composition $\pi_3 \circ \delta: [0,1] \to \mathbb{R}^{n_2}$ is non-constant.
\item If the o-minimal structure admits analytic cell decomposition, the restriction of $\delta$ to $(0,1)$ is real analytic.
\end{enumerate}
\end{thm}

\begin{proof}
This is a special case of Corollary 7.2 in \cite{MR3552014} with $k = \mathbb{Q}$, $\mathbb{R}^n = \mathbb{R}^{n_1} \times \mathbb{R}^{n_2}$ and
\[ \Sigma = \{(y,z_1,z_2) \in Z\mbox{; } y=(y_1,\hdots,y_m) \in \mathbb{Q}^m\mbox{, } \max_{j=1,\hdots,m}H(y_j) \leq T\}.\]
A priori, the corollary only provides $\delta$ such that $(\pi_2,\pi_3) \circ \delta$ is non-constant. Going through its proof, we see however that $\delta$ can actually be chosen such that $\pi_3 \circÊ\delta$ is non-constant. Note that we don't need the additional uniformity in families that the corollary provides.
\end{proof}

\section{Definability}\label{sec:uniformization}
In order to be able to use the powerful o-minimality result from the last section, we must show that our analytic uniformization of $\mathfrak{A}_{g,l}(\mathbb{C})$ is definable, when restricted to a suitable set. In order to be able to speak of e.g. definable or semialgebraic subsets of $\mathbb{C}$ or $\mathbb{H}_g$, we will always identify $\mathbb{C}$ with $\mathbb{R}^2$ and $\M_g(\mathbb{C})$ with $\mathbb{R}^{2g^2}$ by identifying $u + v \sqrt{-1} \in \mathbb{C}$ with $(u,v) \in \mathbb{R}^2$. The following important proposition is due to Peterzil-Starchenko. Note that $\mathbb{P}^{l^g-1}(\mathbb{C})$ is a definable space with respect to its standard atlas.

\begin{prop}\label{prop:uniformization} (Peterzil-Starchenko)
The map $\exp: \mathbb{H}_g \times \mathbb{C}^g  \to \mathfrak{A}_{g,l}(\mathbb{C}) \subset  \mathbb{P}^{l^g-1}(\mathbb{C}) \times \mathbb{P}^{l^g-1}(\mathbb{C})$, defined as in Proposition \ref{prop:universalfamily}, has the following properties:
\begin{enumerate}[label=(\roman*)]
\item There is an open subset $U$ of $\mathbb{H}_g \times \mathbb{C}^g$ such that the restriction of $\exp$ to $U$ is a morphism of definable spaces in $\mathbb{R}_{\an,\exp}$ and $U$ contains the set
\[\{(\tau,\Omega_\tau x) \in F \times \mathbb{C}^g;\\
x \in [0,1)^{2g}\},\]
where $F$ is a Siegel fundamental domain for the congruence subgroup $G(l,2l)$ of $\Sp_{2g}(\mathbb{Z})$.
\item The map $\exp|_{U}$ is surjective.
\end{enumerate}
\end{prop}

\begin{proof}
Going back to the proof of Proposition \ref{prop:universalfamily}, we see that it suffices to show that the function $\phi$ as defined there is definable, when restricted to an open set that contains
\[ \{ (l\tau,\Omega_\tau lx) \in \mathbb{H}_g \times \mathbb{C}^g; \tau \in F, x \in [0,1)^{2g} \}.\]
This is a consequence of Corollary 7.10(1) of \cite{MR3039679} with $D=lE_g$, since $F$ consists of finitely many translates of the Siegel fundamental domain and $l\tau = M[\tau]$, where
\[ M= \begin{pmatrix} \sqrt{l}E_g & 0 \\ 0 & \frac{1}{\sqrt{l}}E_g \end{pmatrix} \in \Sp_{2g}(\mathbb{R}).\]
As $\exp$ is clearly continous, we deduce (i).

Next, we deduce (ii) from Proposition \ref{prop:universalfamily}(ii), since $U$ contains at least one element of each orbit of the action of $G(l,2l) \ltimes \mathbb{Z}^{2g}$ on $\mathbb{H}_g \times \mathbb{C}^g$.
\end{proof}

\section{Functional transcendence}\label{sec:functionaltranscendence}

Let $\mathcal{S} \subset A_{g,l}$ be an irreducible smooth locally closed curve, set $\mathcal{A} = \pi^{-1}(\mathcal{S})$ and let $\mathcal{C}Ê\subset \mathcal{A}$ be an irreducible closed curve. Let $\exp$ be as in Proposition \ref{prop:universalfamily}. Once we have used the Habegger-Pila theorem to find a semialgebraic obstruction, the following theorem (known as ``Ax of log type'') which is due to Gao will allow us to conclude that $\mathcal{C}$ is contained in an irreducible variety as described in Theorem \ref{thm:mainmain} of suitable codimension.

Recall that $\xi$ is the generic point of $\mathcal{S}$ and $\left(\mathcal{A}_\xi^{\overline{\bar{\mathbb{Q}}(\mathcal{S})}/\bar{\mathbb{Q}}},\Tr\right)$ is the $\overline{\bar{\mathbb{Q}}(\mathcal{S})}/\bar{\mathbb{Q}}$-trace of $\mathcal{A}_\xi$. In this section, we will use subscripts to denote the base change of varieties and morphisms.

\begin{thm}\label{lem:ax-lindemann-schanuel} (Gao)
Suppose that $\pi(\mathcal{C}) = \mathcal{S}$. Let $\tilde{Y}$ be an arbitrary complex analytic irreducible component of $\exp^{-1}(\mathcal{C}(\mathbb{C}))$. Then exactly one complex analytic irreducible component of the intersection of the Zariski closure of $\tilde{Y}$ in $\M_g(\mathbb{C}) \times \mathbb{C}^g$ with $\mathbb{H}_g \timesÊ\mathbb{C}^g$ contains $\tilde{Y}$. Furthermore, the intersection of this component with $\exp^{-1}(\mathcal{A}(\mathbb{C}))$ maps under $\exp$ onto the complex points of a subvariety $\mathcal{W}$ of $\mathcal{A}$ containing $\mathcal{C}$ such that over $\overline{\bar{\mathbb{Q}}(\mathcal{S})}$, every irreducible component of $\mathcal{W}_\xi$ is a translate of an abelian subvariety of $\mathcal{A}_\xi$ by a point in $(\mathcal{A}_\xi)_{\tors} + \Tr\left(\mathcal{A}_\xi^{\overline{\bar{\mathbb{Q}}(\mathcal{S})}/\bar{\mathbb{Q}}}(\bar{\mathbb{Q}})\right)$.
\end{thm}

One irreducible component of the variety $\mathcal{W}$ in Theorem \ref{lem:ax-lindemann-schanuel} is the variety the existence of which Theorem \ref{thm:mainmain} postulates. Our statement of the theorem differs from Gao's in the terminology that we use. Before we can prove that our version follows from Gao's version, we need to introduce Gao's terminology: We follow the exposition in \cite{G182}. An abelian subscheme $\mathcal{B}$ of $\mathcal{A}_{\mathbb{C}}$ is an irreducible subgroup scheme of $\mathcal{A}_{\mathbb{C}}Ê\to \mathcal{S}_{\mathbb{C}}$ which is proper and flat over $\mathcal{S}_{\mathbb{C}}$ and dominates $\mathcal{S}_{\mathbb{C}}$. An irreducible subvariety $\mathcal{Z}$ of $\mathcal{A}_{\mathbb{C}}$ is called a generically special subvariety of sg type if there exists a finite cover $\mathcal{S}' \to \mathcal{S}_{\mathbb{C}}$, inducing a morphism $\rho:Ê\mathcal{A}' = \mathcal{A}_{\mathbb{C}} \times_{\mathcal{S}_{\mathbb{C}}} \mathcal{S}' \to \mathcal{A}_{\mathbb{C}}$ such that $\mathcal{Z} = \rho(\sigma'+\sigma_0'+\mathcal{B}')$, where $\mathcal{B}'$ is an abelian subscheme of $\mathcal{A}'$, $\sigma'$ is a torsion section of $\mathcal{A}'$ and $\sigma_0'$ is a constant section of $\mathcal{A}'$, i.e. the composition of a section $\mathcal{S}' \to C' \times \mathcal{S}', s \mapsto (q,s)$ ($C'$ an abelian variety over $\mathbb{C}$, $q \in C'(\mathbb{C})$) with an isomorphism between $C' \times \mathcal{S}'$ and an abelian subscheme of $\mathcal{A}'$. We can now prove Theorem \ref{lem:ax-lindemann-schanuel}.

\begin{proof}
We apply Theorem 8.1 of \cite{G152} to the connected mixed Shimura variety $S = \mathfrak{A}_{g,l}(\mathbb{C})$ with uniformization map $\exp: \mathbb{H}_g \times \mathbb{C}^g \to \mathfrak{A}_{g,l}(\mathbb{C})$ and subvariety $Y$ equal to the Zariski closure of $\mathcal{C}(\mathbb{C})$ in $\mathfrak{A}_{g,l}(\mathbb{C})$. As $Y(\mathbb{C})\backslashÊ\mathcal{C}(\mathbb{C})$ is finite, the closure with respect to the Euclidean topology of $\tilde{Y}$ is a complex analytic irreducible component of $\exp^{-1}(Y(\mathbb{C}))$. Together with Theorem 8.1 of \cite{G152}, this implies that exactly one complex analytic irreducible component of the intersection of the Zariski closure of $\tilde{Y}$ in $\M_g(\mathbb{C}) \times \mathbb{C}^g$ with $\mathbb{H}_g \timesÊ\mathbb{C}^g$ contains $\tilde{Y}$ and that this component maps onto the complex points of a weakly special subvariety $\tilde{\mathcal{W}}$ of $\left(\mathfrak{A}_{g,l}\right)_{\mathbb{C}}$ and that $\tilde{\mathcal{W}}$ is the smallest weakly special subvariety containing $Y$. As a weakly special subvariety, $\tilde{\mathcal{W}}$ is irreducible. By Proposition 5.3 of \cite{G182} (cf. Proposition 1.1 and 3.3 of \cite{G15}) the variety $\tilde{\mathcal{W}}$ is a generically special subvariety of sg type of the abelian scheme $\pi^{-1}(\pi(\tilde{\mathcal{W}})) \to \pi(\tilde{\mathcal{W}})$, where this term is defined analogously for $\pi^{-1}(\pi(\tilde{\mathcal{W}})) \to \pi(\tilde{\mathcal{W}})$ as for $\mathcal{A}_{\mathbb{C}} \toÊ\mathcal{S}_{\mathbb{C}}$ (see Definition 1.5 in \cite{G182}).

A priori, $\tilde{\mathcal{W}}$ is defined over $\mathbb{C}$, but since it is the smallest such weakly special subvariety, Galois conjugates of weakly special subvarieties as well as irreducible components of intersections of weakly special subvarieties are weakly special and $\mathcal{C}$ and hence $Y$ are defined over $\bar{\mathbb{Q}}$, it must be defined over $\bar{\mathbb{Q}}$. We set $\mathcal{W} = \tilde{\mathcal{W}} \cap \mathcal{A}$, considered as a variety over $\bar{\mathbb{Q}}$. This is a subvariety of $\mathcal{A}$ that contains $\mathcal{C}$.

Let $L$ be an algebraic closure of the function field of $\mathcal{S}_{\mathbb{C}}$. We identify $\overline{\bar{\mathbb{Q}}(\mathcal{S})}$ with the algebraic closure of $\bar{\mathbb{Q}}(\mathcal{S})$ in $L$. The $L/\mathbb{C}$-trace of $\left(\mathcal{A}_\xi\right)_L$ coincides with the base change of the $\mathbb{C}\overline{\bar{\mathbb{Q}}(\mathcal{S})}/\mathbb{C}$-trace of $\left(\mathcal{A}_\xi\right)_{\mathbb{C}\overline{\bar{\mathbb{Q}}(\mathcal{S})}}$, which coincides with the base change of the $\overline{\bar{\mathbb{Q}}(\mathcal{S})}/\bar{\mathbb{Q}}$-trace of $\mathcal{A}_\xi$ by Theorem 6.8 of \cite{MR2255529}. As $\tilde{\mathcal{W}}$ is generically special of sg type (as a variety over $\mathbb{C}$), it follows from the universal property of the trace that every irreducible component of $\left(\mathcal{W}_\xi\right)_L$ is a translate of an abelian subvariety of $\left(\mathcal{A}_\xi\right)_L$ by a point in $\left(\mathcal{A}_\xi\right)_{\tors}+\Tr_L\left(\mathcal{A}_\xi^{\overline{\bar{\mathbb{Q}}(\mathcal{S})}/\bar{\mathbb{Q}}}(\mathbb{C})\right)$.

For every torsion point $t$ of $\mathcal{A}_\xi$ the subvariety $\Tr_L^{-1}\left(\left(t+\mathcal{W}_{\xi}\right)_L\right)$ of $\left(\mathcal{A}^{\overline{\bar{\mathbb{Q}}(\mathcal{S})}/\bar{\mathbb{Q}}}_\xi\right)_L$ is then defined both over $\mathbb{C}$ and $\overline{\bar{\mathbb{Q}}(\mathcal{S})}$, since all abelian subvarieties and torsion points of $\left(\mathcal{A}^{\overline{\bar{\mathbb{Q}}(\mathcal{S})}/\bar{\mathbb{Q}}}_\xi\right)_L$ are defined over $\bar{\mathbb{Q}}$. Hence, $\Tr_L^{-1}\left(\left(t+\mathcal{W}_{\xi}\right)_L\right)$ is defined over the intersection of these two fields in $L$, which is equal to $\bar{\mathbb{Q}}$. Therefore, the irreducible components of $\mathcal{W}_\xi$ are in fact translates of abelian subvarieties by points in $\left(\mathcal{A}_\xi\right)_{\tors}+\Tr\left(\mathcal{A}_\xi^{\overline{\bar{\mathbb{Q}}(\mathcal{S})}/\bar{\mathbb{Q}}}(\bar{\mathbb{Q}})\right)$. The theorem follows.

Note that Proposition 5.3 of \cite{G182} applies only to the universal family of principally polarized abelian varieties with symplectic level $l$-structure, but the same statement can be proved analogously for any connected mixed Shimura variety of Kuga type coming from a neat congruence subgroup, so in particular for $\mathfrak{A}_{g,l}(\mathbb{C})$ (see Proposition 1.2.14 and Corollary 1.2.15 in Gao's dissertation \cite{GDiss}). One could also apply Proposition 5.3 of \cite{G182} to an irreducible component of the preimage of $\tilde{\mathcal{W}}(\mathbb{C})$ under the canonical Shimura morphism from the universal family of principally polarized abelian varieties with symplectic level $2l$-structure to $\mathfrak{A}_{g,l}(\mathbb{C})$.
\end{proof}

\section{Proof of Theorem \ref{thm:mainmain}, Theorem \ref{thm:main} and Corollary \ref{cor:dacorollary}}\label{sec:mainproof}

\subsection{Proof of Theorem \ref{thm:mainmain}}

We assume that $\mathcal{A}_\Gamma^{[k]} \cap \mathcal{C}$ is infinite and want to show that $\mathcal{C}$ is contained in an irreducible subvariety $\mathcal{W}$ of the form described in Theorem \ref{thm:mainmain}.

\subsubsection{Reduction to the universal family}

\begin{lem}
We can assume without loss of generality that $\mathcal{S} \subset A_{g,l}$ is a smooth irreducible locally closed curve (not necessarily closed in $A_{g,l}$) and $\mathcal{A} = \pi^{-1}(\mathcal{S})$.
\end{lem}

\begin{proof}
For $l$ big enough, the scheme $A_{g,l}$ with the family of abelian varieties $\mathfrak{A}_{g,l} \to A_{g,l}$ is the fine moduli scheme of principally polarized abelian varieties of dimension $g$ with level structure ``between $l$ and $2l$''. For the precise moduli interpretation, see \cite{MR1304906}, Appendix to Chapter 7, Section B. In particular, if our family is a pull-back of the universal family of principally polarized abelian varieties of dimension $g$ with symplectic level $2l$-structure, it will automatically also be a pull-back of $\mathfrak{A}_{g,l} \to A_{g,l}$.

Let therefore $\mathcal{A} \to \mathcal{S}$ for the moment be an arbitrary abelian scheme over an irreducible smooth curve of relative dimension $g$. If $\xi$ is the generic point of $\mathcal{S}$, then the abelian variety $\mathcal{A}_\xi$ is isogenous to a principally polarized abelian variety $\tilde{A}$. The abelian variety $\tilde{A}$ as well as the isogeny are defined over some finite extension $F$ of $\bar{\mathbb{Q}}(\mathcal{S})$. After replacing $\mathcal{S}$ by a finite cover $\mathcal{S}' \to \mathcal{S}$ and $\mathcal{A}$ by its pullback under that cover, we may assume that $F = \bar{\mathbb{Q}}(\mathcal{S})$. We can replace $\mathcal{S}$ by a finite cover, since an irreducible subvariety $\mathcal{W} \subset \mathcal{A} \times_{\mathcal{S}} \mathcal{S'}$ as described in Theorem \ref{thm:mainmain} projects to an irreducible subvariety of $\mathcal{A}$ of the same form. By Theorem 3 of Section 1.4 in \cite{MR1045822}, there exists a N\'{e}ron model $\tilde{\mathcal{A}}$ of $\tilde{A}$ over $\mathcal{S}$ as defined in Definition 1 of Section 1.2 in \cite{MR1045822}. By the universal property of the N\'{e}ron model, we obtain an $\mathcal{S}$-morphism $\mathcal{A} \to \tilde{\mathcal{A}}$ which extends the isogeny between $\mathcal{A}_\xi$ and $\tilde{A}$.

By Theorem 3 of Section 1.4 in \cite{MR1045822}, there is a Zariski open subset $\tilde{\mathcal{S}}$ of $\mathcal{S}$ such that $\tilde{\mathcal{A}}Ê\times_{\mathcal{S}} \tilde{\mathcal{S}}$ is an abelian scheme over $\tilde{\mathcal{S}}$. Since $\tilde{\mathcal{S}}$ is smooth, it follows as in \cite{MR1083353}, p. 6, that the abelian scheme $\tilde{\mathcal{A}}Ê\times_{\mathcal{S}} \tilde{\mathcal{S}}$ is principally polarized, i.e. admits an isomorphism of group schemes over $\tilde{\mathcal{S}}$ to its dual abelian scheme such that the restriction of the isomorphism to each fiber over a geometric point of $\tilde{\mathcal{S}}$ is induced by an ample line bundle on that fiber. The morphism between $\mathcal{A}$ and $\tilde{\mathcal{A}}$ that extends the isogeny between $\mathcal{A}_\xi$ and $\tilde{A}$ is dominant and proper, hence surjective, so its restriction to each fiber over a point in $\tilde{\mathcal{S}}$ is an isogeny. We see that it suffices to prove the theorem for $\tilde{\mathcal{A}}Ê\times_{\mathcal{S}} \tilde{\mathcal{S}}\toÊ\tilde{\mathcal{S}}$, hence we can assume that $\mathcal{A}$ is a principally polarized abelian scheme.

We can then add symplectic level $2l$-structure to the family $\mathcal{A} \to \mathcal{S}$ by taking a finite cover of $\mathcal{S}$ (corresponding to the finite field extension of $\bar{\mathbb{Q}}(\mathcal{S})$ that is obtained by adding the $2l$-torsion points of the generic fiber).

Having done this, there is a cartesian diagram
\[
\begin{tikzcd}
\mathcal{A} \arrow{r}{i} \arrow[swap]{d}{} & \arrow{d}{} \mathfrak{A}_{g,l} \\
\mathcal{S} \arrow{r}{i_{\mathcal{S}}} & A_{g,l}
\end{tikzcd},
\]
where the morphisms $i$ and $i_{\mathcal{S}}$ are defined over $\bar{\mathbb{Q}}$. This is a consequence of Theorem 7.9 in \cite{MR1304906} that asserts the existence of a fine moduli space for principally polarized abelian varieties of dimension $g$ with full level $l$-structure for $l$ big enough (in fact, $lÊ\geq 3$ suffices). The family $\mathcal{A}$ is then a pullback of the universal family with symplectic level $2l$-structure and therefore also of the family $\mathfrak{A}_{g,l} \to A_{g,l}$ (cf. \cite{MR1304906}, Appendix to Chapter 7, Section B). For every $s \in \mathcal{S}$, the restriction $i|_{\mathcal{A}_s}$ is an isomorphism between $\mathcal{A}_s$ and $i(\mathcal{A}_s)$.

If the family $\mathcal{A}$ is not isotrivial, as we suppose in our theorem, the map $i_{\mathcal{S}}$ is non-constant, so has finite fibers, and therefore $i$ has finite fibers as well. Thus, the curve $i(\mathcal{C})$ must intersect the enlarged isogeny orbit in infinitely many points as well. If $\mathcal{W} \subset i(\mathcal{A})$ is of the form described in Theorem \ref{thm:mainmain}, then every irreducible component of $i^{-1}(\mathcal{W}) \subset \mathcal{A}$ that dominates $\mathcal{S}$ is as well, so it suffices to prove our theorem for $i(\mathcal{A}) \to i_{\mathcal{S}}(\mathcal{S})$. We can even pass to a Zariski open smooth subset of $i_{\mathcal{S}}(\mathcal{S})$ (we use that $i(\mathcal{C})$ intersects every fiber in only finitely many points). This proves the lemma.

\end{proof}

\subsubsection{Producing many points of bounded height}
We now return to subfamilies of $\mathfrak{A}_{g,l} \to A_{g,l}$ of the form $\pi^{-1}(\mathcal{S}) \to \mathcal{S}$ with $\mathcal{S}$ smooth, irreducible and locally closed. We will keep the same notation until the end of the proof.

We have
\[\sup_{p \in \mathcal{A}_\Gamma^{[k]} \cap \mathcal{C}}{[K(p):K]}=\infty,\]
since otherwise $\{\pi(p); p \in \mathcal{A}_\Gamma^{[k]}\cap\mathcal{C}\}$ would be a subset of $\mathcal{S}$ of bounded degree and hence bounded height by Lemma \ref{lem:heightbound}. By Northcott's theorem, this set would be finite and hence $\mathcal{A}_\Gamma^{[k]}\cap\mathcal{C}$ would be finite as well, since $\mathcal{C}$ intersects every fiber of $\pi$ in only finitely many points.

For each $s \in \mathcal{S}$ such that $A_0$ and $\mathcal{A}_s$ are isogenous, let $\phi_s: A_0 \to \mathcal{A}_s$ be the isogeny furnished by Corollary \ref{cor:choiceofisogeny}. We choose a point $p \in \mathcal{A}_\Gamma^{[k]} \cap \mathcal{C}$. Thanks to Lemma \ref{lem:technicallemma}, we can write $p=\phi_{\pi(p)}(\gamma+b)$ for some $\gamma \in \Gamma$, $b \in B_0$ with $B_0$ an abelian subvariety of $A_0$ of codimension $\geq k$. We set $s=\pi(p)$, $d=[K(p):K]$. By the above, we can make $d$ arbitrarily big with the right choice of $p$. If $\sigma$ is an element of $\Gal(\bar{\mathbb{Q}}/K)$, then it follows that $\sigma(p)=\sigma(\phi_s)(\sigma(\gamma)+\sigma(b))$, where $\sigma$ acts on algebraic points and maps in the usual way.

As $\mathcal{C}$ and $\mathcal{S}$ are defined over $K$, the points $\sigma(p)$ and $\sigma(s)$ lie again on $\mathcal{C}$ and $\mathcal{S}$ respectively. Note that the addition morphism $A_0 \times A_0 \to A_0$ and the inversion morphism $A_0 \to A_0$ are both defined over $K$ -- in particular, $\sigma$ fixes the zero element of $A_0$. Furthermore, it sends the zero element of $\mathcal{A}_s$ to the zero element of $\mathcal{A}_{\sigma(s)}$. It follows that the map $\sigma(\phi_s)$ is an isogeny between $\sigma(A_0)=A_0$ and $\sigma(\mathcal{A}_s) = \mathcal{A}_{\sigma(s)}$ with kernel $\sigma(\ker \phi_s)$ and therefore has degree $\deg \sigma(\phi_s)=\deg\phi_s$. Since we have assumed that all endomorphisms of $A_0$ are defined over $K$, we have $\sigma(B_0) = B_0$.

Finally, if $N \in \mathbb{N}$ is minimal such that $N\gamma=a_1\gamma_1+\hdots+a_r\gamma_r$ with rational integers $a_1, \hdots, a_r$, then
\[N\sigma(\gamma)=\sigma(N\gamma)=a_1\sigma(\gamma_1)+\hdots+a_r\sigma(\gamma_r)=a_1\gamma_1+\hdots+a_r\gamma_r.\]
It follows that $\sigma(\gamma) \in \Gamma$ and hence $\sigma(p) \in \mathcal{A}_\Gamma^{[k]} \cap \mathcal{C}$. By Lemma \ref{lem:technicallemma}, there exist $\gamma_\sigma \in \Gamma$, an abelian subvariety $B_\sigma$ of $A_0$ of codimension $\geq k$ and $b_\sigma \in B_\sigma$ such that $\phi_{\sigma(s)}(\gamma_\sigma+b_\sigma) = \sigma(p)$, where $\phi_{\sigma(s)}$ is the isogeny chosen in Corollary \ref{cor:choiceofisogeny} and $\deg \phi_{\sigma(s)} \leq \deg \sigma(\phi_s) = \deg \phi_s$. Indeed, we must have $\deg \phi_{\sigma(s)} = \deg \phi_s$, since otherwise $\sigma^{-1}\left(\phi_{\sigma(s)}\right)$ would be an isogeny between $A_0$ and $\mathcal{A}_s$ of degree less than $\deg \phi_s$, a contradiction. We can choose $\gamma_\sigma$ and $b_\sigma$ as in Proposition \ref{prop:galoisbounds}.

Thus, we get $d$ different points $\sigma(p)$ in $\mathcal{A}_\Gamma^{[k]}\cap\mathcal{C}$. Each of these points has some pre-image $(\tau_\sigma,p_\sigma)$ in $U$ under $\exp|_U$ because of Proposition \ref{prop:uniformization}(ii), where $\exp$ and $U$ are defined as in that same proposition. From the proof of Proposition \ref{prop:uniformization}(ii), we see that we can choose $\tau_\sigma$ in a Siegel fundamental domain for $G(l,2l)$ and $p_\sigma$ in a corresponding fundamental parallelogram for the lattice $\tau_\sigma\mathbb{Z}^g+\mathbb{Z}^g$, i.e. $p_\sigma = \Omega_{\tau_\sigma}x_\sigma$ with $x_\sigma \in [0,1)^{2g}$.

The isogeny $\phi_{\sigma(s)}$ pulls back under $\exp(\tau_\sigma,\cdot)$ and $\exp_0$ to a linear map from $\mathbb{C}^g$ to itself, given by some matrix $\alpha_\sigma \in \GL_g(\mathbb{C})$ such that $\alpha_\sigma(\Omega_{\tau_0}\mathbb{Z}^{2g}) \subset \Omega_{\tau_\sigma}\mathbb{Z}^{2g}$ is a subgroup of index $\deg \phi_{\sigma(s)} =Ê\degÊ\phi_s$.

Therefore, there is a matrix $\beta_\sigma \in \M_{2g}(\mathbb{Z}) \cap \GL_{2g}(\mathbb{Q})$ (the rational representation of $\phi_{\sigma(s)}$ with respect to the given uniformizations) satisfying
\[\begin{pmatrix}\alpha_\sigma & 0 \\ 0 & \overline{\alpha_\sigma}\end{pmatrix}\begin{pmatrix} \Omega_{\tau_0} \\ \Omega_{\overline{\tau_0}}\end{pmatrix}=\begin{pmatrix} \Omega_{\tau_\sigma} \\ \Omega_{\overline{\tau_\sigma}}\end{pmatrix}\beta_\sigma.\]
We have $\deg \phi_{\sigma(s)} = |\Delta_\sigma|$, where $\Delta_\sigma := \det \beta_\sigma$.

In fact, the determinant is positive, as it follows from the above that
\[|\det \alpha_\sigma|^2(2\sqrt{-1})^g \det(\Im \tau_0)=(2\sqrt{-1})^g\det(\Im \tau_\sigma)(\det \beta_\sigma).\]
Therefore, we get $\Delta_\sigma=\deg \phi_{\sigma(s)}=\deg \phi_s$ and $\Delta := \Delta_\sigma$ is independent of $\sigma$.

We can write
\[ \beta_\sigma=\begin{pmatrix} \beta_{\sigma,1} & \beta_{\sigma,2} \\ \beta_{\sigma,3} & \beta_{\sigma,4} \end{pmatrix}\]
with $\beta_{\sigma,j} \in \M_g(\mathbb{Z})$ ($j=1,\hdots,4$).
It then follows from the above that
\begin{equation}\label{eq:invertedmatrix}
\alpha_\sigma^{-1}\Omega_{\tau_\sigma}=\Omega_{\tau_0}(\beta_\sigma)^{-1}
\end{equation}
and that
\[\alpha_\sigma \tau_0= \tau_\sigma\beta_{\sigma,1}+\beta_{\sigma,3}, \quad \alpha_\sigma = \tau_\sigma\beta_{\sigma,2}+\beta_{\sigma,4},\]
whence we obtain
\begin{equation}\label{eq:tautotauzero}
\tau_0 = (\tau_\sigma\beta_{\sigma,1}+\beta_{\sigma,3})^t(\tau_\sigma\beta_{\sigma,2}+\beta_{\sigma,4})^{-t} = (\beta_{\sigma,1}^t\tau_\sigma+\beta_{\sigma,3}^t)(\beta_{\sigma,2}^t\tau_\sigma+\beta_{\sigma,4}^t)^{-1}.
\end{equation}

The point $p_\sigma \in \mathbb{C}^g$ satisfies
\[ \exp_0(\alpha_\sigma^{-1}p_\sigma) \in \phi_{\sigma(s)}^{-1}(\sigma(p))=\gamma_\sigma+b_\sigma+\ker \phi_{\sigma(s)}\]
and it follows thanks to $|\ker \phi_{\sigma(s)}|=\deg \phi_{\sigma(s)}=\Delta$ that
\[ \exp_0(N_\sigma\Delta\alpha_\sigma^{-1}p_\sigma)=N_\sigma\Delta\gamma_\sigma+N_\sigma\Delta b_\sigma=\Delta\left(a_{\sigma,1}\gamma_1+\hdots+a_{\sigma,r}\gamma_r\right)+N_\sigma\Delta b_\sigma,\]
where $N_\sigma \in \mathbb{N}$ is minimal such that $N_\sigma\gamma_\sigma \in \mathbb{Z}\gamma_1+\hdots+\mathbb{Z}\gamma_r$ and $a_{\sigma,1},\hdots,a_{\sigma,r} \in \mathbb{Z}$.

As the kernel of $\exp_0$ is $\Omega_{\tau_0}\mathbb{Z}^{2g}$, we deduce that
\begin{equation}\label{eq:isogenyequation}
\Delta(N_\sigma\alpha_\sigma^{-1}p_\sigma-\tilde{\gamma}_\sigma-N_\sigma\tilde{b}_\sigma)= \Omega_{\tau_0} R_\sigma,
\end{equation}
where $R_{\sigma} \in \mathbb{Z}^{2g}$,  $\tilde{b}_\sigma = \Omega_{\tau_0}\tilde{y}_\sigma$ satisfies $\exp_0(\tilde{b}_\sigma) = b_\sigma$ ($\tilde{y}_\sigma \in [0,1)^{2g}$) and $\tilde{\gamma}_\sigma=a_{\sigma,1}\tilde{\gamma_1}+\hdots+a_{\sigma,r}\tilde{\gamma_r}$ with $\tilde{\gamma}_i = \Omega_{\tau_0} u_i$, $u_i \in [0,1)^{2g}$ and $\exp_0(\tilde{\gamma}_i)=\gamma_i$ ($i=1,\hdots,r$).

It now follows from Proposition \ref{prop:galoisbounds}(ii) and Lemma \ref{lem:degreeabeliansubvariety} that there exist a matrix $H_\sigma \in \M_{2g \times 2(g-k)}(\mathbb{Z})$ and $y_\sigma \in [0,1)^{2(g-k)}$ such that $\tilde{y}_\sigma - H_\sigma y_\sigma \in \mathbb{Z}^{2g}$, $\Omega_{\tau_0} H_\sigma$ has rank at most $g-k$ and
\begin{equation}\label{eq:boundforh}
\lVert H_\sigma \rVert \preceq [K(p):K] = d.
\end{equation}
After replacing $R_\sigma$ by $R_\sigma + \Delta N_\sigma(\tilde{y}_\sigma - H_\sigma y_\sigma)$, we can assume that $\tilde{y}_\sigma = H_\sigma y_\sigma$. Of course, we then no longer necessarily have $\tilde{y}_\sigma \in [0,1)^{2g}$.

\begin{lem}\label{lem:complexitybound}
With notation as above, we have
\[\max\{|a_{\sigma,1}|,\hdots,|a_{\sigma,r}|,N_\sigma,\lVert R_\sigma \rVert,\lVert \beta_{\sigma} \rVert, \lVert H_{\sigma}\rVert\} \preceq d\]
for every $\sigma \in \Gal(\bar{\mathbb{Q}}/K)$.
\end{lem}

\begin{proof}
The bound for $\lVert H_{\sigma}Ê\rVert$ has just been established. It follows from Corollary \ref{cor:choiceofisogeny}(ii) and Proposition \ref{prop:galoisbounds}(i) that
\begin{equation}\label{eq:rowsumbound}
\lVert \beta_\sigma \rVert \preceq \deg \phi_{\sigma(s)} = \deg \phi_s \preceq d.
\end{equation}

The matrix $\left(\begin{smallmatrix} \Omega_{\tau_0} \\ \Omega_{\overline{\tau_0}} \end{smallmatrix}\right)$ is invertible and we have
\begin{equation}\label{eq:maximumversuseuclidean}
\lVert R_\sigma \rVert \leq \left\lVert \begin{pmatrix} \Omega_{\tau_0} \\ \Omega_{\overline{\tau_0}} \end{pmatrix}^{-1} \right\rVert \left\lVert \begin{pmatrix} \Omega_{\tau_0} \\ \Omega_{\overline{\tau_0}} \end{pmatrix} R_\sigma \right\rVert.
\end{equation}
It also follows from \eqref{eq:isogenyequation} that
\begin{equation}\label{eq:boundforr}
\lVert \Omega_{\tau_0} R_\sigma \rVert = \lVert \Omega_{\overline{\tau_0}} R_\sigma \rVert \leq \Delta\left(N_\sigma\lVert\alpha_\sigma^{-1}p_\sigma\rVert+(N_\sigma\lVert H_\sigmaÊ\rVert+|a_{\sigma,1}|+\hdots+|a_{\sigma,r}|)\lVert \Omega_{\tau_0}\rVert\right),
\end{equation}
since $\tilde{b}_\sigma = \Omega_{\tau_0}H_\sigma y_\sigma$ with $y_\sigma \in [0,1)^{2(g-k)}$ and $\tilde{\gamma}_i = \Omega_{\tau_0} u_i$ with $u_i \in [0,1)^{2g}$ ($i=1,\hdots,r$).

Furthermore, we know that
\[\alpha_\sigma^{-1}p_\sigma=\alpha_\sigma^{-1}\Omega_{\tau_\sigma}x_\sigma\]
with $x_\sigma \in [0,1)^{2g}$. Using \eqref{eq:invertedmatrix}, we deduce that
\[\alpha_\sigma^{-1}p_\sigma=\Omega_{\tau_0}(\beta_\sigma)^{-1}x_\sigma.\]
Therefore, we can estimate very crudely
\begin{equation}\label{eq:boundforp}
\lVert\alpha_\sigma^{-1}p_\sigma\rVert \leq \lVert\Omega_{\tau_0}\rVert\lVert(\beta_\sigma)^{-1}\rVert\lVert x_\sigma\rVert \preceq \lVert\beta_\sigma\rVert.
\end{equation}

We know thanks to Proposition \ref{prop:galoisbounds}(iii) that
\[ \max\{|a_{\sigma,1}|,\hdots,|a_{\sigma,r}|,N_\sigma\} \preceq [K(\sigma(p)):K] = [K(p):K] = d.\]
Combining this with \eqref{eq:boundforh}, \eqref{eq:rowsumbound}, \eqref{eq:maximumversuseuclidean}, \eqref{eq:boundforr} and \eqref{eq:boundforp}, we deduce that
\begin{equation}\label{eq:essentialbound}
\lVert R_\sigma \rVert \preceq d,
\end{equation}
where the implicit constants are independent of $p$, $\sigma$ and $d$.
\end{proof}

\subsubsection{Application of the point-counting theorem}
From now on, ``definable'' will always mean ``definable in the o-minimal structure $\mathbb{R}_{\an,\exp}$''. Let $\exp$ and $U$ be defined as in Proposition \ref{prop:uniformization}. The set $X=\exp|_U^{-1}(\mathcal{C}(\mathbb{C})) \subset \mathbb{H}_g \timesÊ\mathbb{C}^g$ is definable as $\mathcal{C}(\mathbb{C})$ is semialgebraic, being a quasiprojective algebraic curve, and $\exp|_U$ is definable by Proposition \ref{prop:uniformization}(i).

\begin{lem}
There exists a non-constant real analytic map $\alpha: (0,1) \to X$ such that the transcendence degree over $\mathbb{C}$ of the field generated by its complex coordinate functions is at most $g-k+1$.
\end{lem}

\begin{proof}
Consider the definable set
\begin{align*}
Z=\{(A_1,\hdots,A_r,M,R,B_1,B_2,B_3,B_4,H,A,y,\tau,x) \in \mathbb{R}^{r+1+2g} \times \M_g(\mathbb{R})^4 \\
\times \M_{2g\times2(g-k)}(\mathbb{R}) \times \GL_g(\mathbb{C}) \times \mathbb{R}^{2(g-k)} \times \mathbb{H}_g \times \mathbb{R}^{2g}; (\tau,\Omega_\tau x) \in X, B=\begin{pmatrix}B_1 & B_2 \\ B_3 & B_4 \end{pmatrix},Ê\\
\det B > 0, \det(B_2^t\tau+B_4^t) \neq 0, \tau_0(B_2^t\tau+B_4^t)= B_1^t \tau+B_3^t, \\
\Omega_{\tau_0}H \mbox{ has rank at most $g-k$}, M > 0, A^{-1}\Omega_\tau = \Omega_{\tau_0}B^{-1} \\
\Omega_{\tau_0}R= (\det B)\left(M(\Omega_{\tau_0}B^{-1}x-\Omega_{\tau_0}Hy)-\left(A_1\tilde{\gamma}_1+\hdots+A_r\tilde{\gamma}_r\right)\right)\}.
\end{align*}
What we have done so far, in particular \eqref{eq:invertedmatrix}, \eqref{eq:tautotauzero} and \eqref{eq:isogenyequation}, shows that $Z$ contains $d$ points
\[\left(a_{\sigma,1},\hdots,a_{\sigma,r},N_{\sigma},R_{\sigma},\beta_{\sigma,1},\beta_{\sigma,2},\beta_{\sigma,3},\beta_{\sigma,4},H_{\sigma},\alpha_\sigma,y_\sigma,\tau_\sigma,x_\sigma
\right)\]
such that $a_{\sigma,1},\hdots,a_{\sigma,r} \in \mathbb{Z}$ and $N_{\sigma} \in \mathbb{N}$, $\beta_{\sigma,1}$, $\beta_{\sigma,2}$, $\beta_{\sigma,3}$, $\beta_{\sigma,4} \in \M_{g}(\mathbb{Z})$, $H_{\sigma} \in \M_{2g \times 2(g-k)}(\mathbb{Z})$ and $R_{\sigma} \in \mathbb{Z}^{2g}$. By Lemma \ref{lem:complexitybound}, there is a constant $\kappa$, independent of $p$, $\sigma$ and $d$, such that
\[\max\{|a_{\sigma,1}|,\hdots,|a_{\sigma,r}|,N_\sigma,\lVert R_\sigma \rVert,\lVert \beta_{\sigma} \rVert, \lVert H_{\sigma}\rVert\} \leq d^{\kappa}=:T\]
for $d$ large enough.

We choose $\epsilon =(2\kappa)^{-1}$ and call the set of these $d$ points $\Sigma$. By making $d$ large enough, we can ensure that $d=|\pi_3(\Sigma)|> c(Z,\epsilon)T^{\epsilon}$, where $c(Z,\epsilon)$ is the constant from Theorem \ref{thm:habegger-pila} and 
$\pi_3$ is the projection $\pi_3: Z \to \mathbb{H}_g \times \mathbb{R}^{2g}$ (note that the $(\tau_\sigma,x_\sigma)$ are all different, since the $\sigma(p)=\exp(\tau_\sigma,\Omega_{\tau_\sigma}x_\sigma)$ are).

In the framework of Theorem \ref{thm:habegger-pila} with projection maps $\pi_1: Z \to \mathbb{R}^{r+1+2g} \times \M_g(\mathbb{R})^4 \timesÊ\M_{2g\times2(g-k)}(\mathbb{R})$, $\pi_2: Z \to \GL_{g}(\mathbb{C}) \times \mathbb{R}^{2(g-k)}$ and $\pi_3$, we have that 
\[Ê\Sigma \subset \{(y,z_1,z_2) \in Z\mbox{; } y=(y_1,\hdots,y_m) \in \mathbb{Q}^m\mbox{, } \max_{j=1,\hdots,m}H(y_j) \leq T\}\]
and $|\pi_3(\Sigma)|> cT^{\epsilon}$. Hence, we can apply Theorem \ref{thm:habegger-pila} to $Z$ and obtain a definable real analytic map
\[\delta: (0,1) \to Z.\]
Furthermore, $\pi_1 \circ \delta$ is semialgebraic and $\pi_3 \circ \delta$ is non-constant.

It follows from
\[ \tau_0 = B^t[\tau] \iff \tau = B^{-t}[\tau_0],\]
and $A\Omega_{\tau_0} = \Omega_\tau B$ that first $\tau \circ \delta$ and then $A \circ \delta$ are semialgebraic as well. Here and in the following, we use variables like $\tau$ and $A$ also for the corresponding coordinate functions on $Z$. We can then use that
\[A^{-1}\Omega_\tau x = \Omega_{\tau_0}B^{-1}x=\frac{1}{M}\left(\frac{1}{\det B}\Omega_{\tau_0} R+A_1\tilde{\gamma_1}+\hdots+A_r\tilde{\gamma_r}\right)+\Omega_{\tau_0}Hy\]
to deduce that
\[ \Omega_\tau x = \frac{1}{M}A\left(\frac{1}{\det B}\Omega_{\tau_0} R+A_1\tilde{\gamma_1}+\hdots+A_r\tilde{\gamma_r}\right)+A\Omega_{\tau_0}Hy.\]

Now $H \circÊ\delta$ is semialgebraic. For each $t \in (0,1)$ the rank of $\Omega_{\tau_0}(H \circÊ\delta)(t)$ is at most $g-k$ by the definition of $Z$. Note that the first summand in the above expression for $\Omega_\tau x$ is semialgebraic when composed with $\delta$. Therefore we can conclude that the transcendence degree over $\mathbb{C}$ of the complex coordinate functions of the real analytic map
\[\alpha=\psi \circ \pi_3 \circ \delta: (0,1) \to X\]
is at most $g-k+1$, where $\psi(\tau,x)=(\tau,\Omega_\tau x)$. Since $\pi_3 \circÊ\delta$ is non-constant, so is $\alpha$.
\end{proof}

Since $\alpha$ is non-constant, we can choose some $tÊ\in (0,1)$, where the derivative of $\alpha$ doesn't vanish. Since the Taylor series of $\alpha$ in $t$ must have positive radius of convergence, we can find some holomorphic map $\tilde{\alpha}: D \to X$ from a small open disk $D$ to $X$ such that $t \in D$ and $\tilde{\alpha}|_{D \cap (0,1)} = \alpha|_{D \cap (0,1)}$. By the identity theorem for holomorphic functions, it follows that the transcendence degree over $\mathbb{C}$ of the coordinate functions of $\tilde{\alpha}$ is at most $g-k+1$ as well.

As the derivative of $\alpha$ doesn't vanish at $t$, the map $\tilde{\alpha}$ is non-constant as well. Since every complex analytic irreducible component of $X$ has complex dimension $1$, it follows by analytic continuation of the algebraic relations between the coordinate functions of $\tilde{\alpha}$ along the corresponding complex analytic irreducible component of $\exp^{-1}(\mathcal{C}(\mathbb{C}))$ that the Zariski closure of this complex analytic irreducible component of $\exp^{-1}(\mathcal{C}(\mathbb{C}))$ inside $\M_g(\mathbb{C}) \times \mathbb{C}^g$ has dimension at most $g-k+1$. Theorem \ref{thm:mainmain} now follows from Theorem \ref{lem:ax-lindemann-schanuel}.

\subsection{Proof of Theorem \ref{thm:main}}
If $\pi(\mathcal{C}) = \mathcal{S}$, we can apply Theorem \ref{thm:mainmain} with $k = g$. If $\pi(\mathcal{C}) \neq \mathcal{S}$, there exists $s \in \mathcal{S}$ such that $\mathcal{C} \subset \mathcal{A}_s$ and now we can apply the non-relative Mordell-Lang conjecture which Raynaud proved in this case in \cite{MR726419} by reducing it to the theorem of Faltings to conclude that $\mathcal{A}_\Gamma \cap \mathcal{C} =\phi_s(\Gamma) \cap \mathcal{C}$ is finite unless $\mathcal{C}$ is equal to a translate of an abelian subvariety by a point of $\phi_s(\Gamma) \subset \mathcal{A}_{\Gamma}$. This argument works for any family $\mathcal{A} \to \mathcal{S}$, not only for subfamilies of $\mathfrak{A}_{g,l} \to A_{g,l}$: We need only Lemma \ref{lem:technicallemma} in order to fix the isogeny and this also holds for any family after maybe enlarging $\Gamma$. If $\mathcal{C}$ is then a translate of an abelian subvariety of $\mathcal{A}_s$, it is a translate of that abelian subvariety by any point on $\mathcal{C}$ and hence also by a point in the isogeny orbit of the original $\Gamma$.

\subsection{Proof of Corollary \ref{cor:dacorollary}}
Suppose that $C \cap (\Sigma \times \Gamma')$ is infinite. Let $\mathcal{S}$ be the smooth locus of $\pr_1(C) \subset A_{g,l}$. We can assume without loss of generality that $\dim \mathcal{S} = 1$, otherwise $\pr_1$ is constant and we are done. Let $\pi: \mathfrak{A}_{g,l} \to A_{g,l}$ be the natural morphism as in Section \ref{sec:preliminaries} and let $\epsilon: A_{g,l} \to \mathfrak{A}_{g,l}$ be the zero section.

We apply Theorem \ref{thm:main} to the non-isotrivial abelian scheme $\mathcal{A} = \pi^{-1}(\mathcal{S}) \times _{\mathcal{S}} (\mathcal{S} \times A)$ over $\mathcal{S}$ with $A_0= B \times A$, $\Gamma$ equal to the division closure of $\{ (p,q) \in B \times A\mbox{; $p$ torsion, } q \in \Gamma'\}$ and $\mathcal{C} = (\epsilon(\mathcal{S})) \times_{\mathcal{S}} \pr_1^{-1}(\mathcal{S}) \subset \mathcal{A}$.

A point $p \in \pr_1^{-1}(\mathcal{S}) \cap (\Sigma \times \Gamma')$ yields a point $q = (\epsilon(\pr_1(p)),p)Ê\in \mathcal{C}$. If $\phi: B \to \left(\mathfrak{A}_{g,l}\right)_{\pr_1(p)}$ is an arbitrary isogeny and $0_B$ denotes the neutral element of $B$, then $q$ is the image of $(0_B,\pr_2(p)) \in \Gamma$ under the isogeny $(\phi,\id_A): BÊ\times A \to \left(\mathfrak{A}_{g,l}\right)_{\pr_1(p)} \times A = \mathcal{A}_{\pr_1(p)}$, so $q \in \mathcal{A}_{\Gamma}$. Since $C \cap (\Sigma \times \Gamma')$ is infinite and $C\backslash\pr_1^{-1}(\mathcal{S})$ is finite, the set $\mathcal{C} \cap \mathcal{A}_{\Gamma}$ is infinite as well.

If $\xi$ denotes the generic point of $\mathcal{S}$, then we have $A \subset \mathcal{A}_\xi^{\overline{\bar{\mathbb{Q}}(\mathcal{S})}/\bar{\mathbb{Q}}}$. Since $\mathcal{C}$ dominates $\mathcal{S}$, it doesn't satisfy condition (i) of Theorem \ref{thm:main}, so it has to satisfy condition (ii). This implies that the projection of $\mathcal{C}$ to $\mathcal{S} \times A$ must be the graph of a constant map $\mathcal{S}Ê\to A$. We deduce that $\pr_2$ is constant.

\section*{Acknowledgements}
I thank my advisor Philipp Habegger for suggesting this problem, for his continuous encouragement and for many helpful and interesting discussions. I thank Fabrizio Barroero, Philipp Habegger and Ga\"{e}l R\'{e}mond for helpful comments on a preliminary version of this article. I thank Fabrizio Barroero for pointing out the connection to Gregorio Baldi's article and Gregorio Baldi, whose article brought the conjecture of Buium and Poonen to my attention. I thank the anonymous referee for their helpful suggestions for improving the exposition. This work was partially supported by the Swiss National Science Foundation as part of the project ``Diophantine Problems, o-Minimality, and Heights", no. 200021\_165525.

\bibliographystyle{abbrv}
\bibliography{Bibliography}

\end{document}